\documentclass{amsart} 

\newtheorem{theorem}{Theorem}[section]
\newtheorem{lemma}[theorem]{Lemma}
\newtheorem{proposition}[theorem]{Proposition}
\newtheorem{corollary}[theorem]{Corollary}

\theoremstyle{definition}
\newtheorem{definition}[theorem]{Definition}
\newtheorem{example}[theorem]{Example}
\newtheorem{remark}[theorem]{Remark}

\allowdisplaybreaks

\newcommand{\Z}{\mathbb{Z}}
\newcommand{\R}{\mathbb{R}}

\newcommand{\e}{\varepsilon}
\newcommand{\de}{\delta}

\makeatletter

\@addtoreset{figure}{section}
\makeatother

\makeatletter

\@addtoreset{table}{section}
\makeatother

\makeatletter
  
  \@addtoreset{equation}{section}
\makeatother

\usepackage{longtable,caption} 
\usepackage[dvipdfmx]{graphicx}
\usepackage{comment} 
\usepackage{lscape}
\usepackage{color}
\usepackage{amsmath,amssymb}
\usepackage{physics}
\usepackage[abs]{overpic}
\usepackage{multirow}

\begin{document}


\title[The intersection polynomials of a long virtual knot I]
{The intersection polynomials of a long virtual knot I: Definitions and properties}

\author[T. NAKAMURA]{Takuji NAKAMURA}
\address{Faculty of Education, 
University of Yamanashi,
Takeda 4-4-37, Kofu, Yamanashi 400-8510, Japan}
\email{takunakamura@yamanashi.ac.jp}

\author[Y. NAKANISHI]{Yasutaka NAKANISHI}
\address{Department of Mathematics, Kobe University, 
Rokkodai-cho 1-1, Nada-ku, Kobe 657-8501, Japan}
\email{nakanisi@math.kobe-u.ac.jp}

\author[S. SATOH]{Shin SATOH}
\address{Department of Mathematics, Kobe University, 
Rokkodai-cho 1-1, Nada-ku, Kobe 657-8501, Japan}
\email{shin@math.kobe-u.ac.jp}

\author[K. WADA]{Kodai WADA}
\address{Department of Mathematics, Kobe University, 
Rokkodai-cho 1-1, Nada-ku, Kobe 657-8501, Japan}
\email{wada@math.kobe-u.ac.jp}

\makeatletter
\@namedef{subjclassname@2020}{%
  \textup{2020} Mathematics Subject Classification}
\makeatother


\thanks{This work was supported by JSPS KAKENHI Grant Numbers 
JP20K03621, JP22K03287, and JP23K12973.}


\begin{abstract} 
We introduce twelve polynomial invariants for long virtual knots, 
called intersection polynomials, extending and 
refining the three intersection polynomials for virtual knots.
They are defined via intersection numbers of cycles on a closed surface, 
considering the order of over- and under-crossings.
We study their fundamental properties 
including behavior under symmetries, 
crossing changes, and concatenation products.
All are finite-type invariants of degree two under crossing changes, 
but not under virtualizations, 
and we examine their relation to the closure and the values at 
$t=1$ of their derivatives.
\end{abstract} 

\maketitle


\section{Introduction}\label{sec1} 

In 1999, Kauffman \cite{Kau} introduced the notion of virtual knots  
as a generalization of classical knots, 
which naturally leads to the concept of long virtual knots. 
In 2000, 
Goussarov, Polyak, and Viro \cite{GPV} 
studied long virtual knots from the point of view of finite-type invariants 
and showed that their theory differs substantially from that of virtual knots. 
Manturov \cite{Man} proved that 
the concatenation product of certain pair of long virtual knots is noncommutative, 
and Silver and Williams~\cite{SW} constructed an infinite family of 
long virtual knots whose closures realize any prescribed virtual knot 
(see also \cite{HNNS2}). 
Polynomial invariants for long virtual knots have also been studied extensively. 
Refer to \cite{Af, BFKK, IL, IKK}, for example.

In a series of papers 
\cite{HNNS1}--\cite{HNNS4}, 
Higa with the first three authors introduced and studied 
three polynomial invariants of a virtual knot, 
called the first, second, and third intersection polynomials. 
Their construction is based on 
the intersection number of a pair of cycles on a closed surface. 
In this paper, 
we employ a similar idea to define twelve polynomial invariants 
of a long virtual knot $K$. 
These are also referred to as the intersection polynomials of $K$, 
and are denoted by 
\[F_{ab}(K;t), \ G_{ab}(K;t), \mbox{ and }
H_{ab}(K;t) \mbox{ for }a,b\in\{0,1\}.\]
In the definition, we distinguish the type of each real crossing 
according to whether the overcrossing or undercrossing 
is encountered first along $K$. 

This paper is the first in a series of two. 
Its aim is to introduce the twelve intersection polynomials 
and to study their fundamental properties. 
These invariants refine the three intersection polynomials 
previously defined for virtual knots. 
In the second paper~\cite{NNSW2} of this series, 
we introduce two supporting genera of a long virtual knot 
and establish their relationship 
with the intersection polynomials. 
We further provide  a characterization of 
each intersection polynomial.

The rest of this paper is organized as follows. 
Section~\ref{sec2} introduces the twelve polynomial invariants. 
We prove their invariance in Section~\ref{sec3}, 
and present examples of computations in Section~\ref{sec4}. 
Section~\ref{sec5} investigates their behavior under symmetries 
such as orientation-reversal 
and mirror-reflection (Theorem~\ref{thm52}). 
Section~\ref{sec6} provides formulae 
for the product of long virtual knots (Theorem~\ref{thm62}). 
In Section~\ref{sec7}, we construct 
invariants under crossing changes from the intersection polynomials 
(Theorems~\ref{thm72}). 
In Sections~\ref{sec8} and \ref{sec9}, we prove that 
the intersection polynomials are finite-type invariants 
of degree two under crossing changes (Theorem~\ref{thm82}), 
but not of any degree under virtualizations (Theorem~\ref{thm92}).  
In Section~\ref{sec10}, 
we study the relationship between 
the intersection polynomials of a long virtual knot 
and those of its closure (Proposition~\ref{prop101}), 
and establish several results on the values at $t=1$ 
of their first and second derivatives (Theorem~\ref{thm103}).


\section{Intersection polynomials}\label{sec2} 

A {\it long virtual knot diagram} is an immersed line 
in a plane ${\Bbb R}^2$ such that 
\begin{itemize}
\item[(i)] 
it is identical to the $x$-axis of ${\Bbb R}^2$ outside some $2$-disk, and 
\item[(ii)] 
it has finitely many 
transversal double points 
called {\it real crossings} or {\it virtual crossings}.
\end{itemize}
Here, a real crossing has an over/under information, 
and a virtual crossing is encircled by a small circle. 
Figure~\ref{fig:diagram} shows an example of a long virtual knot diagram 
with five real crossings $c_1,\dots,c_5$ and two virtual crossings. 
A {\it long virtual knot} is an equivalence class of 
such diagrams under generalized Reidemeister moves I--VII. 
Throughout this paper, 
all long virtual knots are oriented from $-\infty$ to $\infty$.

\begin{figure}[htbp]
  \centering
    \begin{overpic}[]{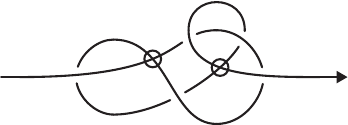}
      \put(-22,21){$-\infty$}
      \put(172,21){$\infty$}
      \put(25,28.5){$c_{1}$}
      \put(77.5,44.5){$c_{2}$}
      \put(122,43.5){$c_{3}$}
      \put(130.5,28.5){$c_{4}$}
      \put(79,2){$c_{5}$}
    \end{overpic}
  \caption{A long virtual knot diagram}
  \label{fig:diagram}
\end{figure}
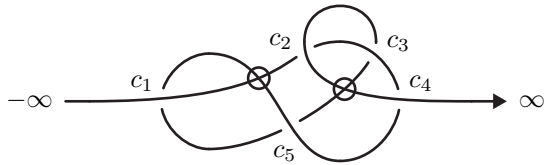

Let $c_1,\dots,c_n$ be the real crossings 
of a long virtual knot diagram $D$, 
which are classified into two types as follows. 
We say that $c_i$ is {\it of type $0$} 
if we pass $c_i$ first going over and then under 
along $D$ from $-\infty$ to $\infty$. 
Otherwise $c_i$ is {\it of type $1$}. 
Refer to \cite{Af,IL,Man}. 
The set of indices of real crossings of type $a\in\{0,1\}$ 
is denoted by 
\[I_a=I_a(D)=\{i\mid \mbox{$c_i$ is of type $a$}\}.\]
For example, the diagram in Figure~\ref{fig:diagram} 
gives $I_0=\{1,3,5\}$ and $I_1=\{2,4\}$.

Let $\Sigma_g$ denote a closed, connected, and oriented surface 
of genus~$g\geq 0$. 
For a long virtual knot diagram $D$, 
a {\it surface realization} of $D$ is 
a pair of $\Sigma_g$ and a knot diagram $D'$ 
with a basepoint such that 
\begin{itemize}
\item[(i)] 
$D'$ lies in $\Sigma_g$ with no virtual crossings, 
\item[(ii)] 
the set of real crossings of $D$ 
coincides with that of $D'$, and 
\item[(iii)] the basepoint of $D'$ corresponds to 
$\pm\infty$ of $D$.
\end{itemize}
Provided there is no ambiguity, we also use the notation $D$ for $D'$. 
A long virtual knot is regarded as an equivalence class of 
such pairs $(\Sigma_g,D)$ under Reidemeister moves I--III and (de)stabilizations 
(cf.~\cite{HNNS2, Kau}). 
For example, the virtual knot diagram $D$ in Figure~\ref{fig:diagram} 
has a surface realization $(\Sigma_2,D)$ 
as shown in Figure~\ref{fig:realization}. 

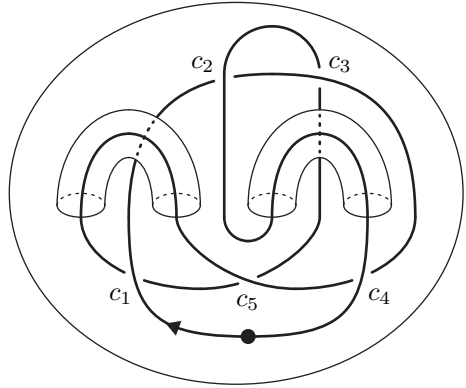
\begin{figure}[htbp]
  \centering
    \begin{overpic}[]{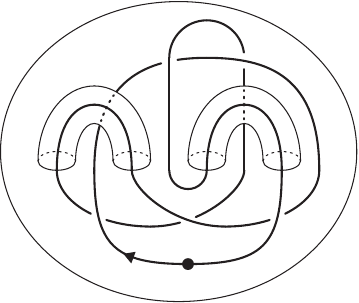}
      \put(38,32){$c_{1}$}
      \put(69,120){$c_{2}$}
      \put(121,120){$c_{3}$}
      \put(136,32){$c_{4}$}
      \put(86,30){$c_{5}$}
    \end{overpic}
  \caption{A surface realization $(\Sigma_2,D)$}
  \label{fig:realization}
\end{figure}

We apply a smoothing operation to a surface realization 
$(\Sigma_g,D)$ at a real crossing $c_{i}$ 
to obtain a pair of oriented loops on $\Sigma_g$. 
We denote by $\alpha_i$ (resp. $\beta_i$) 
the loop that does not contain (resp. contains) 
the basepoint of $D$. 
These loops are regarded as homology cycles on $\Sigma_g$. 
As an example, consider a surface realization $(\Sigma_1,D)$ 
with two real crossings $c_1$ and $c_2$ as shown on 
the left of Figure~\ref{fig:cycles}. 
The cycles $\alpha_1$ and $\beta_1$ are 
as depicted in the center of the figure, 
while $\alpha_2$ and $\beta_2$ are shown on the right.

\begin{figure}[htbp]
  \centering
  \vspace{0.5em}
    \begin{overpic}[]{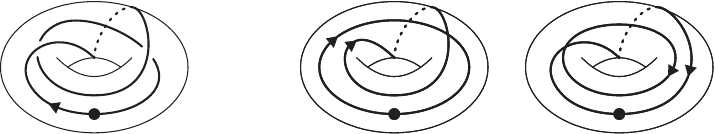}
      \put(8,43){$c_{1}$}
      \put(76,39){$c_{2}$}
      \put(97,24.85){$\xrightarrow[\text{at $c_1,c_2$}]{\text{smoothing}}$}
            \put(205,65){$\alpha_{1}$}
      \put(219.5,9){$\nwarrow$}
      \put(230,0){$\beta_{1}$}
      \put(291,56.7){$\alpha_{2}$}
      \put(327.5,9){$\nwarrow$}
      \put(338,0){$\beta_{2}$}
          \end{overpic}
  \caption{A surface realization $(\Sigma_1,D)$ with four cycles}
  \label{fig:cycles}
\end{figure}
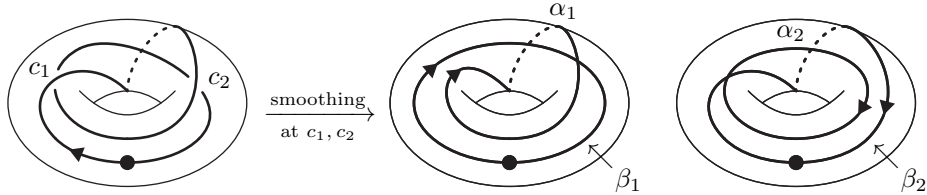

For two cycles $\gamma, \gamma'\in \{\alpha_{1},\dots,\alpha_{n}, 
\beta_{1},\dots,\beta_{n}\}$, 
we denote by $\gamma\cdot\gamma'\in\Z$ 
the intersection number of $\gamma$ and $\gamma'$, 
which is invariant under (de)stabilization of $\Sigma_{g}$. 
We denote by $\varepsilon_i=\varepsilon(c_i)$ 
the sign of a real crossing $c_i$. 

\begin{theorem}[{\cite[Lemma 4.1]{HNNS2}}]\label{thm21}
Let $D$ be a diagram of a long virtual knot $K$, 
and $c_1,\dots,c_n$ the real crossings of $D$. 
For any $a\in\{0,1\}$, the Laurent polynomial 
\[W_a(D;t)=\sum_{i\in I_a}
\varepsilon_i(t^{\alpha_i\cdot\beta_i}-1)\in{\Z}[t,t^{-1}]\]
is an invariant of $K$. 
\qed
\end{theorem}

The Laurent polynomial in Theorem~\ref{thm21} is called 
the {\it $a$-writhe polynomial} of $K$, 
and is denoted by $W_a(K;t)$ 
for $a\in\{0,1\}$. 
For example, since $(\Sigma_1,D)$ in Figure \ref{fig:cycles} satisfies 
$I_0=\{1,2\}$, $I_1=\emptyset$, 
$\e_1=\e_2=1$, $\alpha_1\cdot\beta_1=-1$, 
and $\alpha_2\cdot\beta_2=1$, 
we have 
\[W_0(K;t)=t-2+t^{-1}\mbox{ and }W_1(K;t)=0.\]

\begin{remark}
Let $\gamma_D$ denote the cycle on $\Sigma_g$ presented by $D$. 
Since $\alpha_i+\beta_i=\gamma_D$ holds, 
we have 
\[\alpha_i\cdot\beta_i=\alpha_i\cdot(\gamma_D-\alpha_i)
=\alpha_i\cdot\gamma_D.\] 
In~\cite{HNNS2}, we used $\alpha_i\cdot\gamma_D$ instead of 
$\alpha_i\cdot\beta_i$ to define the $0$- and $1$-writhe polynomials. 
\end{remark}

For each $a,b\in \{0,1\}$, we put 
\begin{align*}
&f_{ab}(D;t)=\sum_{i\in I_a,\,j\in I_b}
\varepsilon_i\varepsilon_j(t^{\alpha_i\cdot\alpha_j}-1),\\ 
&g_{ab}(D;t)=\sum_{i\in I_a,\,j\in I_b}
\varepsilon_i\varepsilon_j(t^{\alpha_i\cdot\beta_j}-1),\mbox{ and }\\
&h_{ab}(D;t)=\sum_{i\in I_a,\,j\in I_b}
\varepsilon_i\varepsilon_j(t^{\beta_i\cdot\beta_j}-1).
\end{align*}
We denote by $\omega_a=\omega_a(D)$ 
the sum of signs of real crossings of $D$ of type $a\in\{0,1\}$; 
that is, 
$\omega_a=\sum_{i\in I_a}\varepsilon_i$. 
We call $\omega_a$ the {\it $a$-writhe} of $D$. 

\begin{theorem}\label{thm22}
Let $D$ be a diagram of a long virtual knot $K$. 
For any $a,b\in\{0,1\}$, the Laurent polynomials 
\begin{align*}
F_{ab}(D;t) &=f_{ab}(D;t), \\
G_{ab}(D;t) &=g_{ab}(D;t)-\omega_b W_a(K;t), \mbox{ and}\\
H_{ab}(D;t) &=h_{ab}(D;t)-\omega_a W_b(K;t)
-\omega_b W_a(K;t^{-1})
\end{align*}
are invariants of $K$. 
\end{theorem}

The twelve Laurent polynomials in Theorem~\ref{thm22} are collectively 
called the {\it intersection polynomials} of a long virtual knot $K$, 
and are denoted by $F_{ab}(K;t)$, $G_{ab}(K;t)$, and $H_{ab}(K;t)$, 
respectively. 
We give the proof of Theorem~\ref{thm22} in Section~\ref{sec3}. 

A Laurent polynomial $f(t)\in{\Z}[t,t^{-1}]$ is called {\it reciprocal} 
if $f(t^{-1})=f(t)$ holds. 

\begin{lemma}\label{lem23}
For a long virtual knot $K$, we have the following. 
\begin{enumerate}
\item
$F_{01}(K;t)=F_{10}(K;t^{-1})$ and $H_{01}(K;t)=H_{10}(K;t^{-1})$. 
\item
$F_{00}(K;t)$, $F_{11}(K;t)$, $H_{00}(K;t)$, and 
$H_{11}(K;t)$ are reciprocal. 
\end{enumerate}
\end{lemma}

\begin{proof}
(i) Since 
$\alpha_i\cdot\alpha_j=-\alpha_j\cdot\alpha_i$ holds 
for any $i$ and $j$, 
we have 
\[
F_{10}(K;t^{-1})=\sum_{i\in I_1,\, j\in I_0}
\varepsilon_i\varepsilon_j(t^{-\alpha_i\cdot\alpha_j}-1)
=\sum_{j\in I_0, \, i\in I_1}
\varepsilon_j\varepsilon_i(t^{\alpha_j\cdot\alpha_i}-1)
=F_{01}(K;t).\]
The other equation can be proved similarly. 

(ii) It holds that 
\[
F_{00}(K;t^{-1})=\sum_{i,j\in I_0}
\varepsilon_i\varepsilon_j(t^{-\alpha_i\cdot\alpha_j}-1)
=\sum_{i,j\in I_0}
\varepsilon_j\varepsilon_i(t^{\alpha_j\cdot\alpha_i}-1)
=F_{00}(K;t).\]
The other equations can be proved similarly. 
\end{proof}

A long virtual knot is called {\it classical} 
if it is presented by a diagram in ${\R}^2$ with no virtual crossings. 

\begin{lemma}\label{lem24}
If $K$ is a long classical knot, 
then we have the following.  
\begin{enumerate}
\item
$W_a(K;t)=0$ for any $a\in\{0,1\}$. 
\item
$X_{ab}(K;t)=0$ 
for any $X\in\{F,G,H\}$ and $a,b\in\{0,1\}$. 
\end{enumerate}
\end{lemma}

\begin{proof}
By definition, $K$ is presented by some surface realization 
$(\Sigma_0,D)$. 
The intersection number of any pair of cycles 
vanishes on a $2$-sphere $\Sigma_0$. 
\end{proof}


\section{Proof of Theorem~\ref{thm22}}\label{sec3} 

Let $(\Sigma_g,D)$ be a surface realization of a long virtual knot $K$. 
To prove Theorem~\ref{thm22}, we prepare Lemmas~\ref{lem31} and \ref{lem32} 
stated below. 

\begin{lemma}\label{lem31}
Let $(\Sigma_g,D')$ be a surface realization obtained from $(\Sigma_g,D)$  
by a Reidemeister move {\rm II} or {\rm III}. 
For any $a,b\in \{0,1\}$, we have 
\[f_{ab}(D;t)=f_{ab}(D';t), \ 
g_{ab}(D;t)=g_{ab}(D';t), \mbox{ and }
h_{ab}(D;t)=h_{ab}(D';t).\]
\end{lemma}

\begin{proof}
The proof of the invariance 
is essentially the same as that of the intersection polynomials for 
a virtual knot, as presented in \cite{HNNS1}. 
For example, assume that $(\Sigma_g,D')$ is obtained from $(\Sigma_g,D)$ 
by a Reidemeister move II canceling a pair of 
real crossings $c_1$ and $c_2$ of $D$. 
We remark that $\alpha_1=\alpha_2$ and  
$\varepsilon_1=-\varepsilon_2$ hold, 
and that $c_1$ and $c_2$ are of the same type. 
If both $c_1$ and $c_2$ are of type $0$, 
then we have 
\begin{align*}
f_{00}(D)-f_{00}(D') &=
\sum_{i,j\in\{1,2\}}\varepsilon_i\varepsilon_j
(t^{\alpha_i\cdot\alpha_j}-1)\\
&\quad+\sum_{j\in I_0(D)\setminus\{1,2\}}
\varepsilon_1\varepsilon_j(t^{\alpha_1\cdot\alpha_j}-1)
+\sum_{j\in I_0(D)\setminus\{1,2\}}
\varepsilon_2\varepsilon_j(t^{\alpha_2\cdot\alpha_j}-1)\\
&\quad+\sum_{i\in I_0(D)\setminus\{1,2\}}
\varepsilon_i\varepsilon_1(t^{\alpha_i\cdot\alpha_1}-1)
+\sum_{i\in I_0(D)\setminus\{1,2\}}
\varepsilon_i\varepsilon_2(t^{\alpha_i\cdot\alpha_2}-1)\\
&=0.
\end{align*}
\end{proof}

\begin{lemma}\label{lem32}
Let $(\Sigma_g,D')$ be a surface realization obtained from $(\Sigma_g,D)$ by 
a Reidemeister move {\rm I} removing 
a real crossing $c_1$ of $D$ 
as shown in {\rm Figure~\ref{R1}}. 
For any $a,b\in \{0,1\}$, we have the following. 
\begin{enumerate}
\item
$f_{ab}(D;t)=f_{ab}(D';t)$. 
\item
$g_{ab}(D;t)-g_{ab}(D';t)=
\begin{cases}
0 & \mbox{for }1\not\in I_b(D), \\
\varepsilon_1 W_a(K;t) & \mbox{for } 1\in I_b(D).
\end{cases}$

\item 
$h_{ab}(D;t)-h_{ab}(D';t)$\\
$=
\begin{cases}
\varepsilon_1(W_b(K;t)+W_a(K;t^{-1})) 
& \mbox{for } a=b\mbox{ and }1\in I_a(D)=I_b(D),\\
0 & \mbox{for }a=b\mbox{ and }1\not\in I_a(D)=I_b(D), \\
\varepsilon_1 W_b(K;t) & \mbox{for } 
a\ne b \mbox{ and } 1\in I_a(D), \\
\varepsilon_1 W_a(K;t^{-1}) & \mbox{for }
a\ne b \mbox{ and }1\in I_b(D). 
\end{cases}$
\end{enumerate}
\end{lemma}

\begin{figure}[htbp]
  \centering
    \begin{overpic}[]{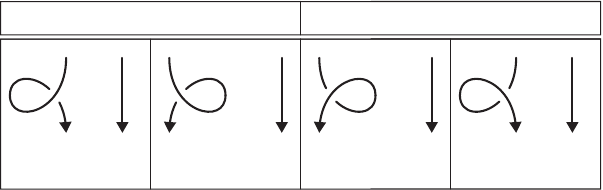}
      \put(52,80){$1\in I_{0}(D)$}
      \put(20,55.5){$c_{1}$}
      \put(41,43){$\to$}
      \put(27,17){$D$}
      \put(54,17){$D'$}
      \put(24,5){$\e_{1}=1$}
      \put(85.5,55.5){$c_{1}$}
      \put(118,43){$\to$}
      \put(76.5,16){$D$}
      \put(130.5,16){$D'$}
      \put(91,5){$\e_{1}=-1$}
      \put(196,80){$1\in I_{1}(D)$}
      \put(157.5,55.5){$c_{1}$}
      \put(189.5,43){$\to$}
      \put(149,16){$D$}
      \put(203,16){$D'$}
      \put(168,5){$\e_{1}=1$}
      \put(236.5,55.5){$c_{1}$}
      \put(257.3,43){$\to$}
      \put(244,16){$D$}
      \put(271,16){$D'$}
      \put(236,5){$\e_{1}=-1$}
    \end{overpic}
  \caption{A Reidemeister move I}
  \label{R1}
\end{figure}
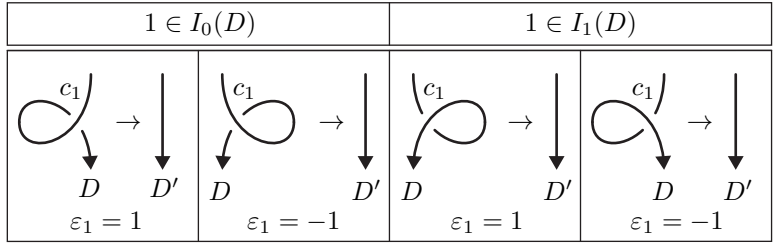

\begin{proof}
The proof is essentially the same as that of the intersection polynomials for 
a virtual knot, as presented in \cite{HNNS1}. 
We prove (iii) only; (i) and (ii) can be proved similarly. 
Since we have $\alpha_1=0$ and $\beta_1=\gamma_D$, 
it holds that 
\[\beta_1\cdot\beta_j=\gamma_D\cdot\beta_j
=(\alpha_j+\beta_j)\cdot\beta_j=\alpha_j\cdot\beta_j.\]

First we consider the case $a=b$. 
If $1\in I_a(D)=I_b(D)$, 
then we have 
\begin{align*}
& h_{aa}(D;t)-h_{aa}(D';t) \\
&= 
\varepsilon_1^2(t^{\beta_1\cdot\beta_1}-1)+
\sum_{j\in I_a(D)\setminus\{1\}}
\varepsilon_1\varepsilon_j
(t^{\beta_1\cdot\beta_j}-1)
+
\sum_{i\in I_a(D)\setminus\{1\}}
\varepsilon_i\varepsilon_1
(t^{\beta_i\cdot\beta_1}-1)\\
&= 
\varepsilon_1\sum_{j\in I_a(D')}
\varepsilon_j
(t^{\alpha_j\cdot\beta_j}-1)
+
\varepsilon_1\sum_{i\in I_a(D')}
\varepsilon_i
(t^{-\alpha_i\cdot\beta_i}-1)\\
&=
\varepsilon_1 W_a(K;t)+
\varepsilon_1 W_a(K;t^{-1}).
\end{align*}
If $1\not\in I_a(D)=I_b(D)$, 
then we clearly have 
$h_{aa}(D;t)=h_{aa}(D';t)$.

Next we consider the case $a\ne b$. 
If $1\in I_a(D)$, 
then we have 
\[h_{ab}(D;t)-h_{ab}(D';t)
=\sum_{j\in I_b(D)}
\varepsilon_1\varepsilon_j(t^{\beta_1\cdot\beta_j}-1)=
\varepsilon_1 W_b(K;t).\]
If $1\in I_b(D)$, 
then 
\[h_{ab}(D;t)-h_{ab}(D';t)
=\sum_{i\in I_a(D)}
\varepsilon_i\varepsilon_1(t^{\beta_i\cdot\beta_1}-1)=
\varepsilon_1 W_a(K;t^{-1}).\]
\end{proof}

\begin{proof}[Proof of {\rm Theorem~\ref{thm22}}.]
We prove only that $H_{ab}(D;t)=H_{ab}(D';t)$ holds 
for any two surface realizations 
$(\Sigma_{g},D)$ and $(\Sigma_{g'},D')$ of $K$. 
The cases of $F_{ab}$ and $G_{ab}$ can be proved similarly. 
Since the intersection numbers $\beta_{i}\cdot\beta_{j}$ are preserved 
by a (de)stabilization of a surface, 
it suffices to consider the case $g=g'$. 
 
Assume that $(\Sigma_g,D')$ is obtained from $(\Sigma_g,D)$ 
by a Reidemeister move II or III. 
Since we have $\omega_a(D)=\omega_a(D')$ and $\omega_b(D)=\omega_b(D')$, 
it follows from Theorem~\ref{thm21} and Lemma~\ref{lem31} that 
\begin{align*}
H_{ab}(D;t) &=h_{ab}(D;t)-\omega_a(D) W_b(K;t)
-\omega_b(D) W_a(K;t^{-1})\\
&=h_{ab}(D';t)-\omega_a(D') W_b(K;t)
-\omega_b(D') W_a(K;t^{-1})
=H_{ab}(D';t).
\end{align*} 

Assume that $(\Sigma_g,D')$ is obtained from $(\Sigma_g,D)$ by a Reidemeister move I 
removing a real crossing $c_1$ of $D$ 
as shown in Figure~\ref{R1}. 
If $a=b$ and $1\in I_a(D)=I_b(D)$, 
then we have 
$\omega_a(D)=\omega_a(D')+\varepsilon_1$ 
and hence
\begin{align*}
H_{aa}(D;t) &=h_{aa}(D;t)-\omega_a(D) W_a(K;t)
-\omega_a(D) W_a(K;t^{-1})\\
&=
h_{aa}(D';t)+\varepsilon_1(W_a(K;t)+W_a(K;t^{-1}))\\
& \quad -(\omega_a(D')+\varepsilon_1)W_a(K;t)
-(\omega_a(D')+\varepsilon_1)W_a(K;t^{-1})\\
&=h_{aa}(D';t)-\omega_a(D') W_a(K;t)
-\omega_a(D') W_a(K;t^{-1})\\
&=H_{aa}(D';t) 
\end{align*}
by Theorem~\ref{thm21} and Lemma~\ref{lem32}(iii). 
The other cases can be proved similarly. 
\end{proof}

\begin{definition}\label{def33}
A long virtual knot diagram $D$ is called 
{\it untwisted} if it satisfies 
$\omega_0(D)=\omega_1(D)=0$. 
\end{definition}

\begin{remark}\label{rem34}
By applying Reidemeister moves of type I appropriately, 
we see that 
any long virtual knot $K$ has an 
untwisted diagram. 
For such a diagram $D$, we have 
\[G_{ab}(K;t)=g_{ab}(D;t)\mbox{ and }H_{ab}(K;t)=h_{ab}(D;t) 
\quad (a,b\in\{0,1\}). 
\] 
\end{remark}


\section{Examples}\label{sec4} 

For an integer $n\geq 2$, 
let $\Gamma(n)$ be a {\it flat} long virtual knot diagram 
with $n+1$ crossings $c_1,\dots,c_{n+1}$ 
as shown in Figure~\ref{fig:ex40}(a), 
where the real crossings have no crossing information.

\begin{figure}[htbp]
  \centering
    \begin{overpic}[]{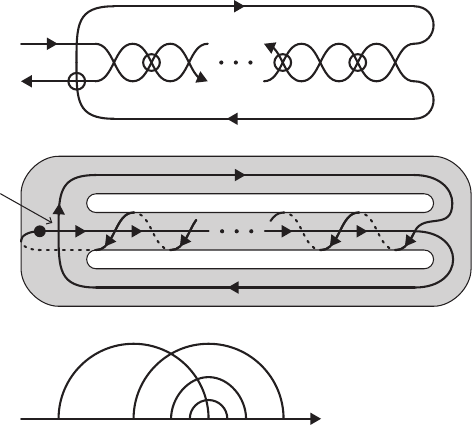}
      \put(-20,171){(a)}
      \put(50,188){$c_{n}$}
      \put(80,188){$c_{n-1}$}
      \put(150,188){$c_{2}$}
      \put(187,188){$c_{1}$}
      \put(14,191){$c_{n+1}$}
       \put(-20,90){(b)}
      \put(189,96.5){$c_{1}$}
      \put(153,96.5){$c_{2}$}
      \put(72,97){$c_{n-1}$}
      \put(45,97){$c_{n}$}
      \put(-12,117){$c_{n+1}$}
      \put(-20,15){(c)}
      \put(15,30){$c_{n+1}$}
      \put(127,32){$c_{n}$}
      \put(113,20){$c_{2}$}
      \put(105.5,12){$c_{1}$}
      \put(24,-7){$-$}
      \put(60.5,-7){$+$}
      \put(68.5,9){$\dots$}
      \put(79,-7){$+$}
      \put(87.75,-7){$+$}
      \put(96.75,-7){$+$}
      \put(105.75,-7){$-$}
      \put(114.75,-7){$-$}
      \put(133,-7){$-$}
      \put(122,9){$\dots$}
     \end{overpic}
  \vspace{1em}
  \caption{The flat long virtual knot diagram $\Gamma(n)$}
  \label{fig:ex40}
\end{figure}
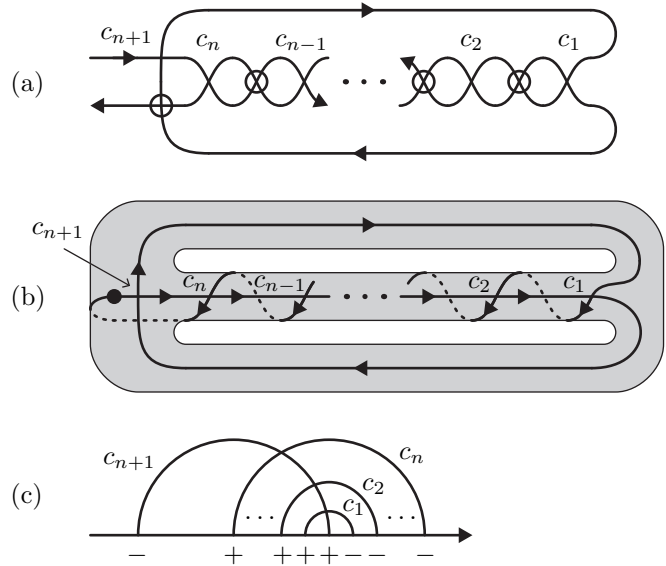

Let $D$ be a long virtual knot diagram 
whose underlying curve is $\Gamma(n)$. 
We remark that the intersection numbers are 
independent of crossing information 
at $c_1,\dots,c_{n+1}$. 
Since $\Gamma(n)$ is realized  
by an immersed circle on $\Sigma_2$ 
as shown in Figure~\ref{fig:ex40}(b), 
we see that 
\[
\alpha_i\cdot\gamma_D=
\begin{cases}
1 & \mbox{for }1\leq i\leq n, \\
n & \mbox{for }i=n+1,
\end{cases}
\]
and 
\[
\alpha_i\cdot\alpha_j=
\begin{cases} 
0 & \mbox{for } 1\leq i,j\leq n \mbox{ and }
i=j=n+1,\\
j-1 & \mbox{for } i=n+1 \mbox{ and }
1\leq j\leq n, \\
-(i-1) & \mbox{for }1\leq i\leq n \mbox{ and }
j=n+1. 
\end{cases}
\]
This is also obtained from the Gauss diagram 
of $\Gamma(n)$ as shown in Figure~\ref{fig:ex40}(c). 
Refer to \cite[Section 3]{HNNS1} for the calculation 
by using a Gauss diagram. 
By the equations 

\begin{align*}
\alpha_i\cdot\beta_j
&=\alpha_i\cdot(\gamma_D-\alpha_j)
=\alpha_i\cdot\gamma_D-
\alpha_i\cdot\alpha_j\mbox{ and}\\
\beta_i\cdot\beta_j
&=
(\gamma_D-\alpha_i)\cdot(\gamma_D-\alpha_j)
=-\alpha_i\cdot\gamma_D+\alpha_j\cdot\gamma_D
+\alpha_i\cdot\alpha_j,
\end{align*}
we have 
\[
\alpha_i\cdot\beta_j=
\begin{cases} 
1 & \mbox{for } 1\leq i,j\leq n,\\
(n+1)-j & \mbox{for } i=n+1 \mbox{ and }
1\leq j\leq n, \\
i & \mbox{for }1\leq i\leq n \mbox{ and }
j=n+1, \\
n & \mbox{for }i=j=n+1,
\end{cases}
\]
and 
\[
\beta_i\cdot\beta_j=
\begin{cases} 
0 & \mbox{for } 1\leq i,j\leq n\mbox{ and }i=j=n+1,\\
-n+j & \mbox{for } i=n+1 \mbox{ and }
1\leq j\leq n, \\
n-i & \mbox{for }1\leq i\leq n \mbox{ and }
j=n+1.
\end{cases}
\]
These calculations are summarized 
in Table~\ref{intersection}. 

\begin{table}[h]
\caption{The intersection numbers for $\Gamma(n)$}
\label{intersection}
\begin{center}
\begin{tabular}{|c||ccc|c|}
\hline
$\alpha_i\cdot\alpha_j$ & 
$1$ & $\cdots$ & $n$ & $n+1$ \\
\hline\hline $1$ & $0$ & $\cdots$ & $0$ & $0$\\
$\vdots$ & $\vdots$ & $\ddots$ & $\vdots$ & $\vdots$ \\
$n$ & $0$ & $\cdots$ & $0$ & $-n+1$ \\
\hline $n+1$ & $0$ & $\cdots$ & $n-1$ & $0$ \\ 
\hline
\end{tabular}
\hspace{8mm}
\begin{tabular}{|c||ccc|c|}
\hline
$\alpha_i\cdot\beta_j$ & 
$1$ & $\cdots$ & $n$ & $n+1$ \\
\hline\hline $1$ & $1$ & $\cdots$ & $1$ & $1$\\
$\vdots$ & $\vdots$ & $\ddots$ & $\vdots$ & $\vdots$ \\
$n$ & $1$ & $\cdots$ & $1$ & $n$ \\
\hline $n+1$ & $n$ & $\cdots$ & $1$ & $n$ \\
\hline
\end{tabular}\\

\vspace{5mm}
\begin{tabular}{|c||ccc|c|}
\hline
$\beta_i\cdot\beta_j$ & 
$1$ & $\cdots$ & $n$ & $n+1$ \\
\hline\hline $1$ & $0$ & $\cdots$ & $0$ & $n-1$\\
$\vdots$ & $\vdots$ & $\ddots$ & $\vdots$ & $\vdots$ \\
$n$ & $0$ & $\cdots$ & $0$ & $0$ \\
\hline $n+1$ & $-n+1$ & $\cdots$ & $0$ & $0$\\ \hline
\end{tabular}
\end{center}
\end{table}

\begin{example}\label{ex41}
Let $D(n)$ be the long virtual knot diagram 
obtained from $\Gamma(n)$ by specifying 
$\e_1,\dots,\e_n=-1$ and $\e_{n+1}=1$. 
Equivalently, every real crossing $c_i$ 
is of type $0$. 
Let $K(n)$ be the long virtual knot 
presented by $D(n)$. 
Since it holds that 
\[I_0=\{1,\dots,n+1\}, \ I_1=\emptyset, \ 
\omega_0=-n+1, \mbox{ and }
\omega_1=0,\]
we have 
\begin{align*}
W_0(K(n);t)
&=\sum_{i=1}^{n+1}
\e_i(t^{\alpha_i\cdot\beta_i}-1)
=t^n-nt+n-1, \\
F_{00}(K(n);t)
&=
\sum_{1\leq i,j\leq n+1}
\e_i\e_j(t^{\alpha_i\cdot\alpha_j}-1)\\
&=
-(t^{n-1}+t^{-n+1})-\dots-(t+t^{-1})+2(n-1),\\
G_{00}(K(n);t)
&=
\sum_{1\leq i,j\leq n+1}
\e_i\e_j(t^{\alpha_i\cdot\beta_j}-1)
-\omega_0 W_0(K(n);t)\\
&=(n-2)t^{n}-2(t^{n-1}+\dots+t)+nt\\
&=
\begin{cases} 
0 & \mbox{for }n=2, \\
(n-2)t^{n}-2(t^{n-1}+\dots+t^2)+(n-2)t & \mbox{for }n\geq 3, 
\end{cases}\\
H_{00}(K(n);t)
&=
\sum_{1\leq i,j\leq n+1}
\e_i\e_j(t^{\beta_i\cdot\beta_j}-1)
-\omega_0 W_0(K(n);t)-\omega_0 W_0(K(n);t^{-1})\\
&=
(n-1)(t^{n}+t^{-n})
-(t^{n-1}+t^{-n+1})-\dots-(t+t^{-1})\\
&\quad -n(n-1)(t+t^{-1})+2n(n-1)\\
&=
\begin{cases}
(t^2+t^{-2})-3(t+t^{-1})+4 & \mbox{for }n=2, \\
(n-1)(t^{n}+t^{-n})-(t^{n-1}+t^{-n+1})-\dots-(t^2+t^{-2}) & \\
-(n^{2}-n+1)(t+t^{-1})+2n(n-1) & \mbox{for }n\geq 3, 
\end{cases}
\end{align*}
and the other polynomials are equal to zero. 
\end{example}

\begin{example}\label{ex42}
Let $D'(n)$ be the long virtual knot diagram 
obtained from $\Gamma(n)$ by specifying 
$\e_1,\dots,\e_{n+1}=1$. 
Equivalently, $c_1,\dots,c_{n}$ are of type $1$ 
and $c_{n+1}$ is of type $0$. 
Let $K'(n)$ be the long virtual knot 
presented by $D'(n)$. 
Since it holds that 
\[I_0=\{n+1\}, \ I_1=\{1,\dots,n\}, \ 
\omega_0=1, \mbox{ and }
\omega_1=n,\]
we have 
\begin{align*}
W_0(K'(n);t)&=
\e_{n+1}(t^{\alpha_{n+1}\cdot\beta_{n+1}}-1)
=t^{n}-1,\\
W_1(K'(n);t)
&=\sum_{i=1}^{n}
\e_i(t^{\alpha_i\cdot\beta_i}-1)
=n(t-1), \\
F_{00}(K'(n);t)
&=\e_{n+1}^2(t^{\alpha_{n+1}\cdot\alpha_{n+1}}-1)=0, \\
F_{01}(K'(n);t)
&=
\sum_{j=1}^{n}
\e_{n+1}\e_j(t^{\alpha_{n+1}\cdot\alpha_j}-1)
=t^{n-1}+\dots+t-n+1,\\
F_{10}(K'(n);t)
&=
F_{01}(K'(n);t^{-1})
=t^{-n+1}+\dots+t^{-1}-n+1,\\
F_{11}(K'(n);t)
&=
\sum_{1\leq i,j\leq n}
\e_i\e_j(t^{\alpha_i\cdot\alpha_j}-1)=0, \\
G_{00}(K'(n);t)
&=
\e_{n+1}^2(t^{\alpha_{n+1}\cdot\beta_{n+1}}-1)
-\omega_0 W_0(K'(n);t)=0, \\
G_{01}(K'(n);t)
&=
\sum_{j=1}^{n}
\e_{n+1}\e_j(t^{\alpha_{n+1}\cdot\beta_j}-1)
-\omega_1 W_0(K'(n);t)\\
&=-(n-1)t^{n}+t^{n-1}+\dots+t,\\
G_{10}(K'(n);t)
&=
\sum_{i=1}^{n}
\e_i\e_{n+1}(t^{\alpha_i\cdot\beta_{n+1}}-1)
-\omega_0 W_1(K'(n);t)
=t^{n}+\dots+t^2-(n-1)t,\\
G_{11}(K'(n);t)
&=
\sum_{1\leq i,j\leq n}
\e_i\e_j(t^{\alpha_i\cdot\beta_j}-1)
-\omega_1 W_1(K'(n);t)=0, \\
H_{00}(K'(n);t)
&=
\e_{n+1}^2(t^{\beta_{n+1}\cdot\beta_{n+1}}-1)
-\omega_0 W_0(K'(n);t)-\omega_0 W_0(K'(n);t^{-1})\\
&=-t^{n}+2-t^{-n}, \\
H_{01}(K'(n);t)
&=
\sum_{j=1}^{n}
\e_{n+1}\e_j(t^{\beta_{n+1}\cdot\beta_j}-1)
-\omega_0 W_1(K'(n);t)-\omega_1 W_0(K'(n);t^{-1})\\
&=
-nt^{-n}+t^{-n+1}+\dots+t^{-1}+(n+1)-nt, \\
H_{10}(K'(n);t)
&=H_{01}(K'(n);t^{-1})=
-nt^{n}+t^{n-1}+\dots+t+(n+1)-nt^{-1},\\
H_{11}(K'(n);t)
&=
\sum_{1\leq i,j\leq n}
\e_i\e_j(t^{\beta_i\cdot\beta_j}-1)
-\omega_1 W_1(K(n);t)-\omega_1 W_1(K(n);t^{-1})\\
&=
-n^{2}(t-2+t^{-1}). 
\end{align*}
\end{example}


\section{Symmetries}\label{sec5} 

In what follows, 
when no ambiguity arises, 
we identify a long virtual knot diagram $D$ 
with its surface realization $(\Sigma_g,D)$. 
We construct three diagrams $D^\#$, $-D$, and $D^*$ 
from $D$ as follows: 
\begin{itemize}
\item
$D^\#$ is obtained from $D$ by switching over/under information
at all real crossings of $D$. 
\item
$-D$ is obtained from $D$ by reversing the orientation of $D$. 
\item
$D^*$ is obtained from $D$ by an 
orientation-reversing homeomorphism of 
${\R}^2$ (or $\Sigma_g$). 
\end{itemize}
Let $K$ be the long virtual knot presented by $D$. 
We denote by $K^\#$, $-K$, and $K^*$ the long virtual knots
presented by  $D^\#$, $-D$, and $D^*$, respectively.

In this section, 
we prove the following two theorems, 
where $a'=1-a$ and $b'=1-b$ for $a,b\in\{0,1\}$. 

\begin{theorem}\label{thm51}
For any $a\in\{0,1\}$, we have the following.
\begin{itemize}
\item[(i)] $W_a(K^\#;t)=-W_{a'}(K;t)$. 
\item[(ii)] $W_a({-K};t)=W_{a'}(K;t)$. 
\item[(iii)] $W_a(K^*;t)=-W_a(K;t^{-1})$. 
\end{itemize}
\end{theorem}

\begin{theorem}\label{thm52}
For any $X\in\{F,G,H\}$ and 
$a,b\in\{0,1\}$, we have the following.
\begin{itemize}
\item[(i)] 
$X_{ab}({K^\#}; t)=X_{ab}(-K;t)=X_{a'b'}(K;t)$. 
\item[(ii)] $X_{ab}({K^*};t)=X_{ab}(K;t^{-1})$.
\end{itemize}
\end{theorem}

For a real crossing $c_i$ of $D$, 
we denote by $c_i^\#$, $c_i^-$, and $c_i^*$ the corresponding real crossings 
of $D^\#$, $-D$, and $D^*$, respectively. 
The sign and the two cycles associated with each of these crossings 
are analogously indicated by 
appending the superscripts $\#$, $-$, and $*$, respectively.

\begin{proof}[Proof of {\rm Theorem~\ref{thm51}}]
(i) Since it holds that 
\[I_a(D^\#)=I_{a'}(D), \ \varepsilon_i^\#=-\varepsilon_i, \ 
\alpha_i^\#=\alpha_i, \mbox{ and } \beta_i^\#=\beta_i,\]
we have 
\[W_a(K^\#;t)=\sum_{i\in I_a(D^\#)}
\varepsilon_i^\#(t^{\alpha^\#_i\cdot\beta_i^\#}-1)
=\sum_{i\in I_{a'}(D)}(-\varepsilon_i)(t^{\alpha_i\cdot\beta_i}-1) 
=-W_{a'}(K;t).\]

(ii) Since it holds that 
\[I_a(-D)=I_{a'}(D), \ \varepsilon_i^-=\varepsilon_i, \ 
\alpha_i^-=-\alpha_i, \mbox{ and } \beta_i^-=-\beta_i,\]
we have 
\begin{align*}
W_a(-K;t)&=\sum_{i\in I_a(-D)}
\varepsilon_i^-(t^{\alpha^-_i\cdot\beta_i^-}-1)
=\sum_{i\in I_{a'}(D)}
\varepsilon_i(t^{(-\alpha_i)\cdot(-\beta_i)}-1) \\
&=\sum_{i\in I_{a'}(D)}
\varepsilon_i(t^{\alpha_i\cdot\beta_i}-1)=W_{a'}(K;t).
\end{align*} 

(iii) Since it holds that 
\[I_a(D^*)=I_a(D), \ 
\varepsilon_i^*=-\varepsilon_i, \mbox{ and }
\alpha_i^*\cdot\beta_i^*=-\alpha_i\cdot\beta_i,\]
we have 
\[W_a(K^*;t)
=\sum_{i\in I_a(D^*)}
\varepsilon_i^*(t^{\alpha^*_i\cdot\beta^*_i}-1) 
=\sum_{i\in I_a(D)}
(-\varepsilon_i)(t^{-\alpha_i\cdot\beta_i}-1) 
=-W_a(K;t^{-1}).\]
\end{proof}

\begin{proof}[Proof of {\rm Theorem~\ref{thm52}}]
We prove only the case $X=H$; 
the other cases can be proved similarly. 
We may assume that $D$ is untwisted, 
and hence so are $D^\#$, $-D$, and $D^*$. 
By Remark~\ref{rem34}, 
we have $H_{ab}(K;t)=h_{ab}(D;t)$. 

(i) Similarly to the proofs of Theorem~\ref{thm51}(i) and (ii), 
we have 
\begin{align*}
H_{ab}(K^\#;t)&=h_{ab}(D^\#;t)
=\sum_{\substack{i\in I_a(D^\#)\\ j\in I_b(D^\#)}}
\varepsilon^\#_i\varepsilon^\#_j(t^{\beta^\#_i\cdot\beta^\#_j}-1)\\
&=\sum_{\substack{i\in I_{a'}(D)\\ j\in I_{b'}(D)}}
(-\varepsilon_i)(-\varepsilon_j)(t^{\beta_i\cdot\beta_j}-1)
=h_{a'b'}(D;t)=H_{a'b'}(K;t) \mbox{ and}\\
H_{ab}(-K;t)&=h_{ab}(-D;t)=\sum_{\substack{i\in I_a(-D)\\ j\in I_b(-D)}}
\varepsilon^-_i\varepsilon^-_j(t^{\beta^-_i\cdot\beta^-_j}-1)\\
&=\sum_{\substack{i\in I_{a'}(D)\\ j\in I_{b'}(D)}}
\varepsilon_i \varepsilon_j(t^{(-\beta_i)\cdot(-\beta_j)}-1)
=h_{a'b'}(D;t)=H_{a'b'}(K;t). 
\end{align*}

(ii) Similarly to the proof of Theorem~\ref{thm51}(iii), 
we have
\begin{align*}
H_{ab}(K^*;t)&=h_{ab}(D^*;t)=\sum_{\substack{i\in I_a(D^*)\\ j\in I_b(D^*)}}
\varepsilon^*_i\varepsilon^*_j(t^{\beta^*_i\cdot\beta^*_j}-1)\\
&=\sum_{\substack{i\in I_a(D)\\ j\in I_b(D)}}
(-\varepsilon_i)(-\varepsilon_j)(t^{-\beta_i\cdot\beta_j}-1)
=h_{ab}(D;t^{-1})=H_{ab}(K;t^{-1}), 
\end{align*}
where we use the equation 
$\beta_i^*\cdot\beta_j^*=-\beta_i\cdot\beta_j$. 
\end{proof}

\begin{proposition}\label{prop53}
There are infinitely many long virtual knots $K$ such that 
\[\pm K, \ \pm K^{\#}, \ \pm K^{*}, \mbox{ and }\pm K^{\# *}\] 
are mutually distinct. 
\end{proposition}

\begin{proof}
For an integer $n\geq2$, 
we consider the long virtual knot $K'(n)$ 
given in Example~\ref{ex42}. 
Since it holds that 
\[
W_0(K'(n);t)=t^{n}-1 \text{ and } 
W_1(K'(n);t)=n(t-1),\] 
the eight writhe polynomials 
\[\pm W_0(K'(n);t), \ \pm W_1(K'(n);t),  \ \pm W_0(K'(n);t^{-1}), \
\mbox{ and }\pm W_1(K'(n);t^{-1})\] are mutually distinct.
Therefore, the conclusion follows from Theorem~\ref{thm51}. 
\end{proof}


\section{The product of long virtual knots}\label{sec6} 

Let $D$ and $D'$ be diagrams of long virtual knots $K$ and $K'$, respectively. 
Concatenating $D'$ after $D$ yields 
a new diagram $D''=D\circ D'$. 
The {\it product} of $K$ and $K'$ is defined as  
the long virtual knot presented by $D''$, 
and is denoted by $K\circ K'$. 
This operation is well-defined; 
that is, it does not depend on the choice of diagrams $D$ and $D'$. 
The aim of this section is to prove the following two theorems. 

\begin{theorem}\label{thm61}
For any $a\in\{0,1\}$, we have 
\[W_a(K\circ K';t)=W_a(K;t)+W_a(K';t).\]
\end{theorem} 

\begin{theorem}\label{thm62}
For any $a,b\in\{0,1\}$, we have 
\begin{enumerate}
\item
$F_{ab}(K\circ K';t)=F_{ab}(K;t)+F_{ab}(K';t)$. 
\item
$G_{ab}(K\circ K';t)=G_{ab}(K;t)+G_{ab}(K';t)$. 
\item
$H_{ab}(K\circ K';t)=H_{ab}(K;t)+H_{ab}(K';t)$\\
\phantom{$H_{ab}(K\circ K';t)=$}
$+W_a(K;t^{-1})W_b(K';t)+W_a(K';t^{-1})W_b(K;t)$. 
\end{enumerate}
\end{theorem}

Let $c_1,\dots,c_n$ and $c_{n+1}'\dots,c_m'$ 
be the real crossings of $D$ and $D'$, respectively. 
We attach primes and double primes to the cycles, 
signs, and index sets of $D'$ and $D''$, 
respectively. 
We remark that 
$I_a''=I_a\cup I_a'$ holds for any $a\in\{0,1\}$. 

\begin{lemma}\label{lem63}
The intersection numbers $\alpha_i''\cdot\alpha_j''$, 
$\alpha_i''\cdot\beta_j''$, and $\beta_i''\cdot\beta_j''$ 
among the cycles from $D''$ 
are given as shown in {\rm Table~\ref{int}}. 
\end{lemma}

\begin{table}[htb]
\caption{The intersection numbers for $D''$} 
\label{int}
\begin{center}
\renewcommand{\arraystretch}{1.3}
\begin{tabular}{|c||c|c|}
\hline
$\alpha_i''\cdot\alpha_j''$ & 
$1\leq j\leq n$ & $n+1\leq j\leq m$ \\
\hline\hline
$1\leq i\leq n$ & $\alpha_i\cdot\alpha_j$ & $0$ \\
\hline
$n+1\leq i\leq m$ & $0$ & $\alpha_i'\cdot\alpha_j'$ \\
\hline
\end{tabular}
\end{center}

\vspace{2mm}
\begin{center}
\renewcommand{\arraystretch}{1.3}
\begin{tabular}{|c||c|c|}
\hline
$\alpha_i''\cdot\beta_j''$ & 
$1\leq j\leq n$ & $n+1\leq j\leq m$ \\
\hline\hline
$1\leq i\leq n$ & 
$\alpha_i\cdot\beta_j$ & $\alpha_i\cdot\beta_i$ \\
\hline
$n+1\leq i\leq m$ & 
$\alpha_i'\cdot\beta_i'$ & $\alpha_i'\cdot\beta_j'$ \\
\hline
\end{tabular}
\end{center}

\vspace{2mm}
\begin{center}
\renewcommand{\arraystretch}{1.3}
\begin{tabular}{|c||c|c|}
\hline
$\beta_i''\cdot\beta_j''$ & 
$1\leq j\leq n$ & $n+1\leq j\leq m$ \\
\hline\hline
$1\leq i\leq n$ & 
$\beta_i\cdot\beta_j$ & 
$-\alpha_i\cdot\beta_i+\alpha_j'\cdot\beta_j'$ \\
\hline
$n+1\leq i\leq m$ & 
$-\alpha_i'\cdot\beta_i'+\alpha_j\cdot\beta_j$ 
& $\beta_i'\cdot\beta_j'$ \\ \hline
\end{tabular}
\end{center}
\end{table}

\begin{proof}
Let $(\Sigma_g,D)$ and $(\Sigma_{g'},D')$ 
be surface realizations of $D$ and $D'$, respectively. 
The connected sum of $\Sigma_g$ and $\Sigma_{g'}$ 
as shown in Figure~\ref{fig:product} provides 
a surface realization $(\Sigma_{g+g'},D'')$. 
Then it holds that 
\begin{align*}
\alpha_i''&=
\begin{cases}
\alpha_i & \mbox{for }1\leq i\leq n,\\
\alpha_i' & \mbox{for }n+1\leq i\leq m,
\end{cases}\\
\beta_i''&=
\begin{cases}
\beta_i+\gamma_{D'} & \mbox{for }1\leq i\leq n,\\
\beta_i'+\gamma_D & \mbox{for }n+1\leq i\leq m,
\end{cases}\\
\gamma_{D''}&=\gamma_D+\gamma_{D'}.
\end{align*}

\begin{figure}[htbp]
  \centering
    \begin{overpic}[]{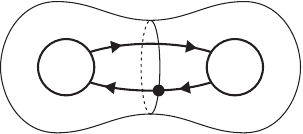}
      \put(27.5,28){$D$}
      \put(108,28){$D'$}
    \end{overpic}
  \vspace{1em}
  \caption{A surface realization $(\Sigma_{g+g'},D'')$}
  \label{fig:product}
\end{figure}

Since the intersection number of 
a cycle on $\Sigma_g$ and a cycle on $\Sigma_{g'}$ 
is equal to zero, 
the conclusion follows: 
For example, for $1\leq i\leq n$ and $n+1\leq j\leq m$, 
we have 
\begin{align*}
\alpha_i''\cdot\alpha_j''&=\alpha_i\cdot\alpha_j'=0,\\
\alpha_i''\cdot\beta_j''&=
\alpha_i\cdot(\beta_j'+\gamma_D)
=\alpha_i\cdot\gamma_D
=\alpha_i\cdot(\alpha_i+\beta_i)=\alpha_i\cdot\beta_i,
\mbox{ and}\\
\beta_i''\cdot\beta_j''&=
(\beta_i+\gamma_{D'})\cdot(\beta_j'+\gamma_D)
=
\beta_i\cdot\gamma_D+\gamma_{D'}\cdot\beta_j'\\
&=\beta_i\cdot(\alpha_i+\beta_i)+(\alpha_j'+\beta_j')\cdot\beta_j'
=-\alpha_i\cdot\beta_i+\alpha_j'\cdot\beta_j'. 
\end{align*}
\end{proof}

\begin{proof}[Proof of {\rm Theorem~\ref{thm61}}.]
By Lemma~\ref{lem63}, we have 
\begin{align*}
W_a(K\circ K';t)&=
\sum_{i\in I_a''}\e_i''(t^{\alpha_i''\cdot\beta_i''}-1)\\
&=\sum_{i\in I_a}\e_i(t^{\alpha_i\cdot\beta_i}-1)
+\sum_{i\in I_a'}\e_i'(t^{\alpha_i'\cdot\beta_i'}-1)\\
&=W_a(K;t)+W_a(K';t). 
\end{align*}
\end{proof}

\begin{proof}[Proof of {\rm Theorem~\ref{thm62}}.]
We may assume that $D$ and $D'$ are untwisted, 
and hence, so is $D''$. 
Then it follows from Lemma~\ref{lem63} that 
\begin{align*}
F_{ab}(K\circ K';t)&=
\sum_{i\in I_a,\, j\in I_b}
\e_i\e_j(t^{\alpha_i\cdot\alpha_j}-1)
+\sum_{i\in I_a',\, j\in I_b'}
\e_i'\e_j'(t^{\alpha_i'\cdot\alpha_j'}-1)\\
&=
F_{ab}(K;t)+F_{ab}(K';t)\mbox{ and}\\
G_{ab}(K\circ K';t)&=
\sum_{i\in I_a,\, j\in I_b}
\e_i\e_j(t^{\alpha_i\cdot\beta_j}-1)
+\sum_{i\in I_a,\, j\in I_b'}
\e_i\e_j'(t^{\alpha_i\cdot\beta_i}-1)\\
&\quad 
+\sum_{i\in I_a',\, j\in I_b}
\e_i'\e_j(t^{\alpha_i'\cdot\beta_i'}-1)
+\sum_{i\in I_a',\, j\in I_b'}
\e_i'\e_j'(t^{\alpha_i'\cdot\beta_j'}-1)\\
&=G_{ab}(K;t)+\omega_b(D')W_a(K;t)
+\omega_{b}(D)W_a(K';t)+G_{ab}(K';t)\\
&=G_{ab}(K;t)+G_{ab}(K';t).
\end{align*}
Moreover, we have 
\begin{align*}
H_{ab}(K\circ K';t)&=
\sum_{i\in I_a,\, j\in I_b}
\e_i\e_j(t^{\beta_i\cdot\beta_j}-1)
+\sum_{i\in I_a,\, j\in I_b'}
\e_i\e_j'(t^{-\alpha_i\cdot\beta_i+\alpha_j'\cdot\beta_j'}-1)\\
&\quad 
+\sum_{i\in I_a',\, j\in I_b}
\e_i'\e_j(t^{-\alpha_i'\cdot\beta_i'+\alpha_j\cdot\beta_j}-1)
+\sum_{i\in I_a',\, j\in I_b'}
\e_i'\e_j'(t^{\beta_i'\cdot\beta_j'}-1). 
\end{align*}
The second sum in the right hand side is equal to 
\begin{align*}
&\sum_{i\in I_a,\, I_b'}\e_i\e_j'(t^{-\alpha_i\cdot\beta_i}-1)
(t^{\alpha_j'\cdot\beta_j'}-1)\\
&+\sum_{i\in I_a,\, j\in I_b'}\e_i\e_j'(t^{-\alpha_i\cdot\beta_i}-1)
+\sum_{i\in I_a,\, j\in I_b'}\e_i\e_j'(t^{\alpha_j'\cdot\beta_j'}-1)\\
=&
W_a(K;t^{-1})W_b(K';t)+\omega_b(D')W_a(K;t^{-1})+\omega_a(D) W_b(K';t)\\
=&W_a(K;t^{-1})W_b(K';t), 
\end{align*}
and the third sum is equal to 
$W_a(K';t^{-1})W_b(K;t)$. 
Therefore, we have 
\begin{align*}
H_{ab}(K\circ K';t)&=
H_{ab}(K;t)+W_a(K;t^{-1})W_b(K';t)\\
&\quad 
+W_a(K';t^{-1})W_b(K;t)+H_{ab}(K';t). 
\end{align*}
\end{proof}

\section{Invariants under crossing changes}\label{sec7}

A {\it crossing change} is a local move 
on a long virtual knot diagram 
that switches the over/under information 
at a real crossing. 
The aim of section is to prove the following 
two theorems. 

\begin{theorem}\label{thm71}
The Laurent polynomial 
\[\widetilde{W}(K;t)=W_0(K;t)-W_1(K;t)\] 
is invariant under crossing changes. 
\end{theorem}

\begin{theorem}\label{thm72}
For any $X\in\{F,G,H\}$, 
the Laurent polynomial 
\[\widetilde{X}(K;t)=X_{00}(K;t)-X_{01}(K;t)-X_{10}(K;t)+X_{11}(K;t)\] 
is invariant under crossing changes. 
\end{theorem}

Let $D$ be a long virtual knot diagram of $K$ 
with $n$ real crossings $c_1,\dots, c_n$ such that 
\begin{itemize}
\item[(i)] 
$c_1,\dots,c_k$ are of type $0$, 
\item[(ii)] 
$c_{k+1},\dots,c_n$ are of type $1$, and 
\item[(iii)] 
the sign of $c_i$ is $\e_i$ $(1\leq i\leq n)$. 
\end{itemize}
Let $D'$ be the diagram obtained from $D$ 
by a crossing change at $c_k$, 
and $K'$ the long virtual knot 
presented by $D'$. 
We remark that 
$c_k$ in $D'$ is of type $1$ 
and has the sign $-\e_k$. 

\begin{proof}[Proof of {\rm Theorem~\ref{thm71}}.]
It follows by assumption that 
\begin{align*}
& I_0(D)=\{1,\dots,k-1,k\}, \ I_1(D)=\{k+1,\dots,n\}, \\
& I_0(D')=\{1,\dots,k-1\}, \mbox{ and }
I_1(D')=\{k,k+1,\dots,n\}.
\end{align*}
Since it holds that 
\begin{align*}
\widetilde{W}(D;t)
&=
\qty(\sum_{i=1}^{k-1}\e_i(t^{\alpha_i\cdot\beta_i}-1)
+\e_k(t^{\alpha_k\cdot\beta_k}-1))
-\sum_{i=k+1}^n\e_i(t^{\alpha_i\cdot\beta_i}-1)\mbox{ and}\\
\widetilde{W}(D';t)
&=
\sum_{i=1}^{k-1}\e_i(t^{\alpha_i\cdot\beta_i}-1)
-\qty((-\e_k)(t^{\alpha_k\cdot\beta_k}-1)
+\sum_{i=k+1}^n\e_i(t^{\alpha_i\cdot\beta_i}-1)), 
\end{align*}
we have $\widetilde{W}(D;t)=\widetilde{W}(D';t)$. 
\end{proof}

\begin{lemma}\label{lem73}
For any $x\in\{f,g,h\}$, 
the Laurent polynomial 
\[\widetilde{x}(D;t)=x_{00}(D;t)-x_{01}(D;t)-x_{10}(D;t)+x_{11}(D;t)\] 
is invariant under crossing changes. 
\end{lemma}

\begin{proof}
We prove only the case $x=g$; 
the other cases can be proved similarly. 

Put $p_{ij}=\e_i\e_j(t^{\alpha_i\cdot\beta_j}-1)$ 
and 
\begin{center}
\renewcommand{\arraystretch}{2}
\begin{tabular}{lll}
$\displaystyle{P_1=
\sum_{\substack{1\leq i\leq k-1\\1\leq j\leq k-1}}p_{ij}}$,
& 
$\displaystyle{P_2=
\sum_{1\leq i\leq k-1}p_{ik}}$,
& 
$\displaystyle{P_3=
\sum_{\substack{1\leq i\leq k-1\\k+1\leq j\leq n}}p_{ij}}$,\\
$\displaystyle{P_4=
\sum_{\substack{1\leq j\leq k-1}}p_{kj}}$,
& 
$\displaystyle{P_5=p_{kk}}$,
& 
$\displaystyle{P_6=
\sum_{k+1\leq j\leq n}p_{kj}}$,\\
$\displaystyle{P_7=
\sum_{\substack{k+1\leq i\leq n\\1\leq j\leq k-1}}p_{ij}}$,
& 
$\displaystyle{P_8=
\sum_{k+1\leq i\leq n}p_{ik}}$,
& 
$\displaystyle{P_9=
\sum_{\substack{k+1\leq i\leq n\\k+1\leq j\leq n}}p_{ij}}$.\\
\end{tabular}
\end{center}
Since $c_k$ in $D$ is of type $0$, 
it holds that 
\begin{center}
\renewcommand{\arraystretch}{1.5}
\begin{tabular}{ll}
$g_{00}(D;t)=P_1+P_2+P_4+P_5$, 
& 
$g_{01}(D;t)=P_3+P_6$, \\
$g_{10}(D;t)=P_7+P_8$, 
& 
$g_{11}(D;t)=P_9$.
\end{tabular}
\end{center}
On the other hand, 
since $c_k$ in $D'$ is of type $1$ 
with the sign $-\e_k$, 
it holds that 
\begin{center}
\renewcommand{\arraystretch}{1.5}
\begin{tabular}{ll}
$g_{00}(D';t)=P_1$, 
& 
$g_{01}(D';t)=-P_2+P_3$, \\
$g_{10}(D';t)=-P_4+P_7$, 
& 
$g_{11}(D';t)=P_5-P_6-P_8+P_9$.
\end{tabular}
\end{center}
Therefore, we have 
$\widetilde{g}(D;t)=\widetilde{g}(D';t)$. 
\end{proof}

\begin{proof}[Proof of {\rm Theorem~\ref{thm72}}.]
We prove only the case $X=G$; 
the other cases can be proved similarly. 
It follows by definition that 
\[\widetilde{G}(K;t)
=\widetilde{g}(D;t)
-\qty(\omega_0(D)-\omega_1(D))
\widetilde{W}(K;t).\]
Since it holds that 
\begin{align*}
\omega_0(D)-\omega_1(D)
&=\qty(\sum_{i=1}^{k-1}\e_i+\e_k)-\sum_{i=k+1}^n\e_i\\
&=\sum_{i=1}^{k-1}\e_i-\qty((-\e_k)+\sum_{i=k+1}^n\e_i)
=\omega_0(D')-\omega_1(D'),
\end{align*}
we have $\widetilde{G}(K;t)=\widetilde{G}(K';t)$ 
by Theorem~\ref{thm71} and 
Lemma~\ref{lem73}. 
\end{proof}

Theorems~\ref{thm71} and \ref{thm72} 
show that $\widetilde{W}$, $\widetilde{F}$, 
$\widetilde{G}$, and $\widetilde{H}$ 
are invariants of {\it flat} long virtual knots. 
Since the long virtual knots $K(n)$ and $K'(n)$ 
given in Section~\ref{sec4} have the same underlying curve $\Gamma(n)$, 
it holds that 
\begin{align*}
W_0(K(n);t)&=W_0(K'(n);t)-W_1(K'(n);t), \\
F_{00}(K(n);t)&=-F_{01}(K'(n);t)-F_{10}(K'(n);t),\\
G_{00}(K(n);t)&=-G_{01}(K'(n);t)-G_{10}(K'(n);t),\mbox{ and}\\
H_{00}(K(n);t)&=H_{00}(K'(n);t)-H_{01}(K'(n);t)-H_{10}(K'(n);t)+H_{11}(K'(n);t).
\end{align*}

A long virtual knot is called {\it descending} 
if it is presented by a descending diagram; 
that is, its all real crossings are of type $0$. 
By performing crossing changes to make a diagram descending, 
we obtain a natural map 
from the set of long virtual knots 
to that of descending long virtual knots (cf.~\cite[Section 2.2]{FKM}). 
Let $K^d$ denote the descending long virtual knot 
associated with a (possibly non-descending) long virtual knot $K$. 
Then we have the following. 

\begin{corollary}\label{cor74}
For a long virtual knot $K$, 
we have the following. 
\begin{enumerate}
\item
$W_a(K^d;t)=
\begin{cases}
W_0(K;t)-W_1(K;t) & \mbox{for }a=0,\\
0 & \mbox{for }a=1. 
\end{cases}$

\item
For any $X\in\{F,G,H\}$, 
\[X_{ab}(K^d;t)=
\begin{cases}
X_{00}(K;t)-X_{01}(K;t)-X_{10}(K;t)+X_{11}(K;t) & \mbox{for }a=b=0, \\
0 & otherwise. 
\end{cases}\]
\end{enumerate}
\end{corollary}

\begin{proof}
Since a descending diagram of $K^d$ 
has no real crossing of type $1$, 
$W_1(K^d;t)=X_{ab}(K^d;t)=0$ 
holds for any $a,b\in\{0,1\}$ except $a=b=0$. 
Moreover, 
since $K$ and $K^d$ are related by a finite sequence of crossing changes, 
it follows from Theorems~\ref{thm71} and \ref{thm72} that 
\[W_0(K^d;t)=W_0(K^d;t)-W_1(K^d;t)=W_0(K;t)-W_1(K;t)\]
and 
\begin{align*}
X_{00}(K^d;t)&=X_{00}(K^d;t)-X_{01}(K^d;t)-X_{10}(K^d;t)+X_{11}(K^d;t)\\
&=X_{00}(K;t)-X_{01}(K;t)-X_{10}(K;t)+X_{11}(K;t).
\end{align*}
\end{proof}

\section{Finite-type invariants under crossing changes}\label{sec8}

There are two definitions of finite-type invariants 
for long virtual knots; one due to Vassiliev \cite{Vas} (see also~\cite{Kau}) 
under crossing changes, 
and the other due to Goussarov, Polyak, and Viro \cite{GPV} 
under virtualizations. 
In this section, we study finite-type invariants under crossing changes, 
and prove that 
the $0$- and $1$-writhe polynomials are finite-type invariants 
of degree one, and the intersection polynomials are 
of degree two. 
It is known that the same property also holds for the writhe polynomial 
and for the intersection polynomials of (closed) virtual knots~\cite{HNNS4}. 

We recall the definition of finite-type invariants 
under crossing changes. 
A {\it $k$-marked} long virtual knot diagram $(D,C_k)$ 
is a pair of a diagram $D$ and a set of 
$k$ specified real crossings $C_k=\{c_1,\dots,c_k\}$ of $D$. 
For $k$ integers $\de_1,\dots,\de_k\in\{0,1\}$, 
we denote by $D_{\de_1,\dots,\de_k}$ 
the long virtual knot diagram obtained from $(D,C_k)$ 
by modifying each $c_i\in C_k$ according to $\de_i\in\{0,1\}$ 
with the rule shown in Figure~\ref{fig:doublept}, 
where we mark the real crossings belonging to $C_k$ 
with an astarisk $*$. 
Let $K$ be the long virtual knot presented by $D$. 
We denote by $K_{\de_1,\dots,\de_k}$ the long virtual knot 
presented by $D_{\de_1,\dots,\de_k}$.

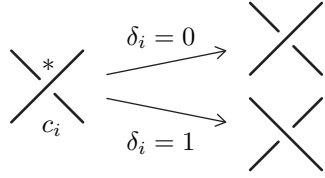
\begin{figure}[htbp]
  \centering
    \begin{overpic}[]{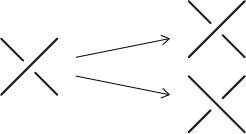}
      \put(12,15){$c_{i}$}
      \put(12,38){$*$}
      \put(44,48){$\de_{i}=0$}
      \put(44,8){$\de_{i}=1$}
    \end{overpic}
  \caption{Performing a crossing change for $\delta_i=1$}
  \label{fig:doublept}
\end{figure}

Let $v$ be an invariant of long virtual knots 
taking values in an abelian group. 
Such an invariant is a finite-type invariant 
of degree $k$ under crossing changes if and only if 
\[v(D,C_{k+1})=\sum_{\de_1,\dots,\de_{k+1} \in\{0,1\}}
(-1)^{\de_1+\dots+\de_{k+1}} v(K_{\de_1,\dots,\de_{k+1}})=0\]
holds for any $(D,C_{k+1})$, 
and there is $(D',C_k)$ with $v(D',C_k)\ne 0$ 
(cf.~\cite{Kau,Vas}).

\begin{theorem}\label{thm81}
For any $a\in \{0,1\}$, 
the $a$-writhe polynomial 
$W_a(K;t)$ is a finite-type invariant of degree one 
under crossing changes. 
\end{theorem}

\begin{proof}
Let $(D,C_{2})$ be a $2$-marked long virtual knot diagram 
with $C_{2}=\{c_{1},c_{2}\}$, 
and $c_3,\ldots,c_n$ the non-marked real crossings of $D$.
For $1\leq i\leq n$ and $\de_1,\de_2\in\{0,1\}$, 
we put \[\Phi_i^{\de_1,\de_2}(t)=
\begin{cases}
\e_i(t^{\alpha_i\cdot\beta_i}-1) & \mbox{for }i\in I_a(D_{\de_1,\de_2}), \\
0 & \mbox{for }i \not\in I_a(D_{\de_1,\de_2}). 
\end{cases}\] 
Here, the sign $\e_i$ of $c_i$ 
and the intersection number $\alpha_i\cdot\beta_i$ 
are taken in $D_{\de_1,\de_2}$. 
Then it holds that 
\[\sum_{\de_1,\de_2\in\{0,1\}}(-1)^{\de_1+\de_2}
\left(\sum_{i\in I_a(D_{\de_1,\de_2})}
\e_i(t^{\alpha_i\cdot\beta_i}-1)\right)
=\sum_{i=1}^n
\left(\sum_{\de_1,\de_2\in\{0,1\}}(-1)^{\de_1+\de_2}
\Phi_i^{\de_1,\de_2}(t)\right).\]
We prove that the second sum in the right hand side 
of the above equation is equal to zero for each $i$ 
with $1\leq i\leq n$. 

For $i=1$, 
it follows by definition that 
the type and sign of $c_1$ in $D_{\de_1,\de_2}$ 
and the intersection number $\alpha_1\cdot\beta_1$ 
are independent of the choice of $\delta_2\in\{0,1\}$. 
Hence, $\Phi_1^{\de_1,\de_2}(t)$ is 
also independent of $\de_2$; 
that is, $\Phi_1^{\de_1,0}(t)=\Phi_{1}^{\de_{1,1}}(t)$. 
Therefore, we have 
\[\sum_{\de_1,\de_2\in\{0,1\}}(-1)^{\de_1+\de_2}
\Phi_1^{\de_1,\de_2}(t)=
\sum_{\de_1\in\{0,1\}}
\left((-1)^{\de_1}\Phi_1^{\de_1,0}(t)+
(-1)^{\de_1+1}\Phi_1^{\de_1,1}(t)\right)=0.\]
Similarly, the second sum for $i=2$ is also equal to zero. 

For each $i$ with $3\leq i\leq n$, 
the type and sign of $c_i$ and the intersection number 
$\alpha_i\cdot\beta_i$ in $D_{\de_1,\de_2}$ 
are independent of both $\de_1$ and $\de_2$. 
Hence, $\Phi_i^{\de_1,\de_2}(t)$ is also independent of $\de_1$ and $\de_2$, 
and the second sum for such $i$ is again equal to zero. 

On the other hand, we consider a $1$-marked long virtual knot 
diagram $(D,C_{1})$ with $C_{1}=\{c_{1}\}$ 
as shown on the left of Figure~\ref{fig:pf-order1}. 
Then $W_a(D_0;t)=t-1$ holds for $a\in\{0,1\}$. 
Furthermore, since $D_1$ presents the trivial long virtual knot, 
$W_a(D_1;t)=0$ holds for $a\in\{0,1\}$. 
See the middle and right of the figure. 
Therefore, we have 
\[W_a(D_0;t)-W_a(D_1;t)=t-1\neq 0,\] 
which shows that $W_a(K;t)$ is a finite-type invariant 
of degree one under crossing changes. 
\end{proof}

\begin{figure}[htbp]
  \centering
    \begin{overpic}[]{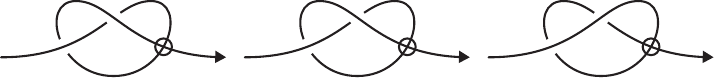}
      \put(51,37){$c_{1}$}
      \put(52.5,19){$*$}
      \put(40,-12){$(D,C_1)$}
      \put(168,-12){$D_0$}
      \put(284,-12){$D_1$}
    \end{overpic}
  \vspace{1em}
  \caption{$(D,C_1)$ with $D_0$ and $D_1$}
  \label{fig:pf-order1}
\end{figure}
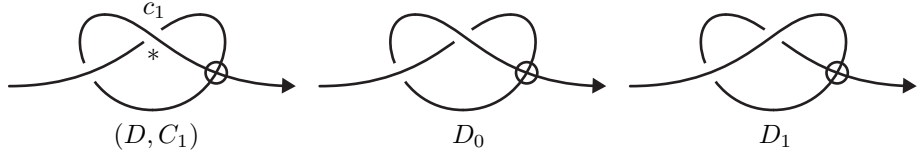

\begin{theorem}\label{thm82}
For any $X\in \{F,G,H\}$ and $a,b\in \{0,1\}$, 
the intersection polynomial 
$X_{ab}(K;t)$ is a finite-type invariant of degree two 
under crossing changes. 
\end{theorem}

To prove Theorem~\ref{thm82},  
we prepare Lemmas~\ref{lem83} and~\ref{lem84} stated below. 

\begin{lemma}\label{lem83}
Let $(D,C_3)$ be a $3$-marked long virtual knot diagram 
with $C_3=\{c_1,c_2,c_3\}$. 
For any $x\in \{f,g,h\}$ and $a,b\in \{0,1\}$, 
we have 
\[
\sum_{\de_1,\de_2,\de_3\in \{0,1\}}
(-1)^{\de_1+\de_2+\de_3}x_{ab}(D_{\de_1,\de_2,\de_3};t)=0.
\]
\end{lemma}

\begin{proof}
We prove only the case $x=f$; 
the other cases can be proved similarly. 
Let $c_4,\dots,c_n$ be the non-marked real crossings of $D$. 
For $1\leq i,j\leq n$ and $\de_1,\de_2,\de_3\in\{0,1\}$, 
we put \[\Psi_{i,j}^{\de_1,\de_2,\de_3}(t)=
\begin{cases}
\e_i\e_j(t^{\alpha_i\cdot\alpha_j}-1) & 
\mbox{for }i\in I_a(D_{\de_1,\de_2,\de_3})\mbox{ and }
j\in I_b(D_{\de_1,\de_2,\de_3}), \\
0 & \mbox{otherwise}. 
\end{cases}\] 
Then it holds that 
\begin{align*}
&\sum_{\de_1,\de_2,\de_3\in\{0,1\}}(-1)^{\de_1+\de_2+\de_3}
\left(\sum_{\substack{i\in I_a(D_{\de_1,\de_2,\de_3})\\ j\in I_b(D_{\de_1,\de_2,\de_3})}}
\e_i\e_j(t^{\alpha_i\cdot\alpha_j}-1)\right)\\
&=\sum_{1\leq i,j\leq n}
\left(\sum_{\de_1,\de_2,\de_3\in\{0,1\}}(-1)^{\de_1+\de_2+\de_3}
\Psi_{i,j}^{\de_1,\de_2,\de_3}(t)\right).
\end{align*}
It suffices to prove that the second sum in the right hand side 
of the above equation is equal to zero for each $(i,j)$ 
with $1\leq i,j\leq n$. 
The proof is similar to that of Theorem~\ref{thm81}. 
In fact, we have the following. 
\begin{itemize}
\item
For $(i,j)$ with $1\leq i,j\leq 3$, 
$\Psi_{i,j}^{\de_1,\de_2,\de_3}(t)$ 
is independent of $\de_\ell$, where $\ell\in\{1,2,3\}\setminus\{i,j\}$. 
\item
For $(i,j)$ with $1\leq i\leq 3$ and $4\leq j\leq n$, 
$\Psi_{i,j}^{\de_1,\de_2,\de_3}(t)$ is independent of $\de_\ell$, 
where $\ell\in\{1,2,3\}\setminus\{i\}$. 
\item
For $(i,j)$ with $4\leq i\leq n$ and $1\leq j\leq 3$, 
$\Psi_{i,j}^{\de_1,\de_2,\de_3}(t)$ is independent of $\de_\ell$, 
where $\ell\in\{1,2,3\}\setminus\{j\}$. 
\item
For each $(i,j)$ with $4\leq i,j\leq n$, 
$\Psi_{i,j}^{\de_1,\de_2,\de_3}(t)$ 
is independent of $\de_1,\de_2,\de_3$. 
\end{itemize}
Therefore, the second sum for each $(i,j)$ is equal to zero. 
\end{proof}

\begin{lemma}\label{lem84}
Let $(D,C_3)$ be a $3$-marked long virtual knot diagram 
with $C_3=\{c_1,c_2,c_3\}$. 
For any $a,b\in \{0,1\}$, we have 
\[\sum_{\de_1,\de_2,\de_3\in \{0,1\}}(-1)^{\de_1+\de_2+\de_3}
\omega_a(D_{\de_1,\de_2,\de_3})W_{b}(K_{\de_1,\de_2,\de_3};t)=0.\]
\end{lemma}

\begin{proof}
Let $c_4,\dots,c_n$ be the non-marked real crossings of $D$. 
For $i$ with $1\leq i\leq n$, we put  
\[\e_i^{\de_1,\de_2,\de_3}=
\begin{cases}
\e_i & \mbox{for }i\in I_a(D_{\de_1,\de_2,\de_3}), \\
0 & \mbox{for }i \not\in I_a(D_{\de_1,\de_2,\de_3}). 
\end{cases}\] 
By $\omega_a(D_{\de_1,\de_2,\de_3})=
\sum_{i=1}^n \e_i^{\de_1,\de_2,\de_3}$, 
it holds that 
\begin{align*}
&\sum_{\de_1,\de_2,\de_3\in \{0,1\}}(-1)^{\de_1+\de_2+\de_3}
\omega_a(D_{\de_1,\de_2,\de_3})W_{b}(K_{\de_1,\de_2,\de_3};t)\\
&=\sum_{i=1}^n\left(
\sum_{\de_1,\de_2,\de_3\in\{0,1\}}(-1)^{\de_1+\de_2+\de_3}
\e_i^{\de_1,\de_2,\de_3}W_b(K_{\de_1,\de_2,\de_3};t)\right).
\end{align*}
We prove that the second sum in the right hand side of the above equation 
is equal to zero for each $i$ with $1\leq i\leq n$. 

For $i=1$, since $\e_1^{\de_1,\de_2,\de_3}$ is 
independent of both $\de_2$ and $\de_3$, 
Theorem~\ref{thm81} induces that the second sum is equal to 
\[\sum_{\de_1\in\{0,1\}}(-1)^{\de_1}\e_1^{\de_1,0,0}
\left(\sum_{\de_2,\de_3\in\{0,1\}}(-1)^{\de_2+\de_3}
W_b(K_{\de_1,\de_2,\de_3};t)\right)=0.\]
Similarly, the second sums for $i=2$ and $3$ are 
also equal to zero. 

For each $i$ with $4\leq i\leq n$, 
$\e_i^{\de_1,\de_2,\de_3}$ is 
independent of $\de_1$, $\de_2$, and $\de_3$. 
Hence, the second sum for such $i$ is again 
equal to zero. 
\end{proof}

\begin{proof}[Proof of {\rm Theorem~\ref{thm82}}.]
By Lemmas~\ref{lem83} and \ref{lem84}, 
we have 
\[
\sum_{\de_{1},\de_{2},\de_{3}\in\{\pm1\}}
(-1)^{\de_1+\de_2+\de_3}X_{ab}(K_{\de_{1},\de_{2},\de_{3}};t)=0
\]
for any $X\in\{F,G,H\}$ and $a,b\in\{0,1\}$.

On the other hand, 
we consider a $2$-marked diagram $(D,C_2)$ 
with $C_{2}=\{c_{1},c_{2}\}$ 
as shown on the top left of Figure~\ref{fig:pf-order2}. 
Then it holds that 
\[\sum_{\de_1,\de_2\in\{0,1\}}(-1)^{\de_1+\de_2}F_{ab}(D_{\de_1,\de_2};t)
=t-2+t^{-1}\ne 0\]
and 
\begin{align*}
\sum_{\de_1,\de_2\in\{0,1\}}(-1)^{\de_1+\de_2}G_{ab}(D_{\de_1,\de_2};t)
&=\sum_{\de_1,\de_2\in\{0,1\}}(-1)^{\de_1+\de_2}H_{ab}(D_{\de_1,\de_2};t)\\
&=-t+2-t^{-1}\ne 0,
\end{align*}
which shows that $X_{ab}(K;t)$ is a finite-type invariant of degree two 
under crossing changes. 
\end{proof}

\begin{figure}[t]
  \centering
    \begin{overpic}[]{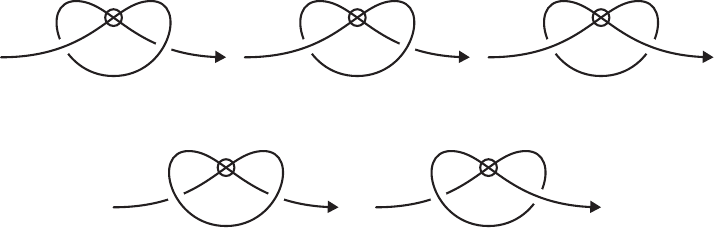}
      \put(19,89){$*$}
      \put(84,89){$*$}
      \put(20,75){$c_{1}$}
      \put(80,75){$c_{2}$}
      \put(40,60){$(D,C_2)$}
      \put(163,60){$D_{0,0}$}
      \put(282,60){$D_{0,1}$}
      \put(99,-12){$D_{1,0}$}
      \put(226,-12){$D_{1,1}$}
    \end{overpic}
  \vspace{1em}
  \caption{$(D,C_2)$ with $D_{0,0}$, $D_{0,1}$, $D_{1,0}$, and $D_{1,1}$}
  \label{fig:pf-order2}
\end{figure}
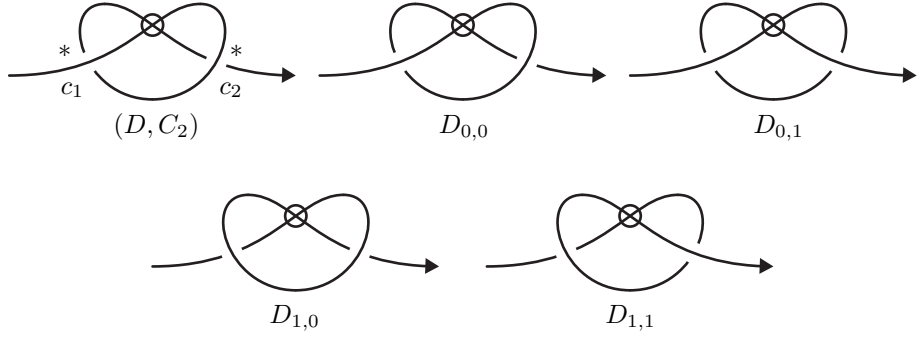

\section{Finite-type invariants under virtualizations}\label{sec9}

In this section, 
we study finite-type invariants under virtualizations. 
The definition is quite similar to that under crossing changes, 
and the only difference lies in whether we use 
virtualizations or crossing changes (cf.~\cite{GPV}). 

Let $(D,C_k)$ be a $k$-marked long virtual knot diagram 
with $C_k=\{c_1,\dots,c_k\}$. 
For $k$ integers $\de_1,\dots,\de_k\in\{0,1\}$, 
we also, by abuse of notation, denote by $D_{\de_1,\dots,\de_k}$ 
the long virtual knot diagram obtained from $(D,C_k)$ 
by modifying $c_i\in C_k$ with the rule 
shown in Figure~\ref{fig:semi-virtual}. 
Let $K_{\de_1,\dots,\de_k}$ be the long virtual knot 
presented by $D_{\de_1,\dots,\de_k}$.

\begin{figure}[htbp]
  \centering
    \begin{overpic}[]{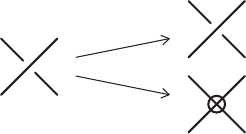}
      \put(12,15){$c_{i}$}
      \put(12,38){$*$}
      \put(44,48){$\de_{i}=0$}
      \put(44,8){$\de_{i}=1$}
    \end{overpic}
  \caption{Performing a virtualization for $\de_i=1$}
  \label{fig:semi-virtual}
\end{figure}
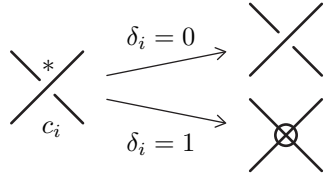

An invariant $v$ of long virtual knots taking values in an abelian group 
is {\it not} a finite-type invariant under virtualizations 
if and only if there exist infinitely many integers $n\geq0$ 
such that there is an $n$-marked diagram $(D(n),C_n)$ 
satisfying 
\[v(D(n),C_{n})=\sum_{\de_1,\dots,\de_n \in\{0,1\}}
(-1)^{\de_1+\dots+\de_n} v(K(n)_{\de_1,\dots,\de_n})\ne 0.\]

\begin{theorem}\label{thm91}
For any $a\in\{0,1\}$, 
the $a$-writhe polynomial $W_{a}(K;t)$ is not a finite-type invariant 
under virtualizations. 
\end{theorem}

\begin{proof}
By Theorem~\ref{thm51}(ii), 
it suffices to consider the $0$-writhe polynomial $W_{0}$. 
Let $D(n)$ $(n\geq2)$ be the long virtual knot diagram 
given in Example~\ref{ex41}. 
We consider the $(n-2)$-marked diagram $(D(n),C_{n-2})$ 
with $C_{n-2}=\{c_1,\dots,c_{n-2}\}$. 
Let $s$ denote the number of $0$'s among 
the integers $\de_1,\dots,\de_{n-2}\in\{0,1\}$. 
Since $D(n)_{\de_1,\dots,\de_{n-2}}$ 
represents the long virtual knot $K(s+2)$, 
we have 
\begin{align*}
&\sum_{\de_1,\dots,\de_{n-2}\in\{0,1\}}
(-1)^{\de_1+\dots+\de_{n-2}}
W_0(D(n)_{\de_1,\dots,\de_{n-2}};t) \\
&=\sum_{s=0}^{n-2} (-1)^{n-2-s}{n-2\choose s}
W_0(K(s+2);t).
\end{align*} 
By the equation  
$W_0(K(s+2);t)=t^{s+2}-(s+2)t+s+1$ given in Example \ref{ex41}, 
the maximal degree of the sum is equal to $n$, 
which completes the proof. 
\end{proof}

\begin{theorem}\label{thm92}
For any $a,b\in\{0,1\}$ and $X\in\{F,G,H\}$, 
the intersection polynomial $X_{ab}(K;t)$ is not a finite-type invariant under virtualizations. 
\end{theorem}

\begin{proof}
For the polynomials $F_{00}$, $G_{00}$, and $H_{00}$, 
the proofs are similar to that of $W_0$ in Theorem~\ref{thm91}. 
In fact, the maximal degrees of 
$F_{00}(K(n);t)$, $G_{00}(K(n);t)$, and $H_{00}(K(n);t)$ $(n\geq 3)$ 
are equal to $n-1$, $n$, and $n$, respectively. 

On the other hand, for the polynomials 
$F_{01}$, $G_{01}$, and $H_{10}$, 
we may use the $(n-2)$-marked diagram 
$(D'(n), C'_{n-2})$ with  $C_{n-2}'=\{c_1,\dots,c_{n-2}\}$, 
where $D'(n)$ is the long virtual knot diagram 
given in Example~\ref{ex42}. 
In this case, the maximal degrees of 
$F_{01}(K'(n);t)$, $G_{01}(K'(n);t)$, and $H_{10}(K'(n);t)$ 
are equal to $n-1$, $n$, and $n$, respectively. 
Therefore, $F_{01}$, $G_{01}$, and $H_{10}$ are not finite-type invariants.

By Theorem~\ref{thm52}(i), the remaining  
intersection polynomials are also not finite-type invariants. 
\end{proof}


\section{The intersection polynomials of a virtual knot}\label{sec10}

For a diagram $D$ of a long virtual knot $K$, 
we denote by $\widehat{D}$ 
the diagram obtained from $D$ 
by identifying $-\infty$ with $\infty$. 
The {\it closure} of $K$ is the virtual knot 
presented by $\widehat{D}$, 
and is denoted by $\widehat{K}$.

We recall the definitions of the writhe and intersection polynomials 
of a virtual knot $\kappa$. 
Let $\Delta$ be a diagram of a virtual knot $\kappa$, 
and $c_1,\dots,c_n$ the real crossings of $\Delta$. 
Consider a surface realization $(\Sigma_{g},\Delta)$ of $\Delta$.  
For each real crossing $c_i$, 
let $\gamma_i$ (resp. $\overline{\gamma}_i$) 
be the cycle on $\Sigma_{g}$ presented by the part of $\Delta$ 
running from the overcrossing to the undercrossing 
(resp. from the undercrossing to the overcrossing) at $c_i$. 
See Figure~\ref{fig:cycles-knot}. 
We denote by $\omega(\Delta)$ the sum of signs 
of real crossings of $\Delta$. 
Then the following Laurent polynomials are invariants of $\kappa$; 
\begin{align*}
&W(\kappa;t)=\sum_{i=1}^n \varepsilon_i
(t^{\gamma_i\cdot\overline{\gamma}_i}-1), \\
&
I(\kappa;t)=\sum_{i,j=1}^n
\varepsilon_i\varepsilon_j(t^{\gamma_i\cdot\overline{\gamma}_j}-1)
-\omega(\Delta) W(\kappa;t), \mbox{ and }\\
&
I\!I(\kappa;t)=\sum_{i,j=1}^n
\varepsilon_i\varepsilon_j(t^{\gamma_i\cdot\gamma_j}-1)
+\sum_{i,j=1}^n
\varepsilon_i\varepsilon_j
(t^{\overline{\gamma}_i\cdot\overline{\gamma}_j}-1)
-\omega(\Delta)\qty(W(\kappa;t)+W(\kappa;t^{-1})).
\end{align*}
These are called the {\it writhe polynomial} \cite{ST}, {\it first intersection polynomial}, and 
{\it second intersection polynomial} \cite{HNNS1} of $\kappa$, respectively. 
We remark that the writhe polynomial is 
also defined in \cite{CG,K3} independently. 

\begin{figure}[t]
  \centering
    \begin{overpic}[]{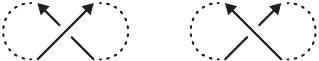}
      \put(28,22){$c_{i}$}
      \put(19,-12){$\e_{i}=1$}
      \put(10,13){$\overline{\gamma}_{i}$}
      \put(45,13){$\gamma_{i}$}
      \put(118.5,22){$c_{i}$}
      \put(101,12){$\gamma_{i}$}
      \put(135.5,12){$\overline{\gamma}_{i}$}
      \put(105.5,-12){$\e_{i}=-1$}
    \end{overpic}
  \vspace{1em}
  \caption{The two cycles $\gamma_{i}$ and $\overline{\gamma}_{i}$ at a real crossing $c_{i}$ of $\Delta$}
  \label{fig:cycles-knot}
\end{figure}
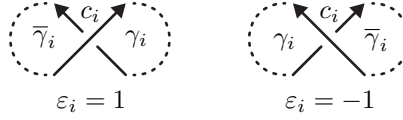

For a long virtual knot $K$, 
the polynomials $W$, $I$, and $I\!I$ of the closure $\widehat{K}$ 
is expressed by using $W_{a}$, $F_{ab}$, $G_{ab}$ and $H_{ab}$ of $K$ 
as follows. 

\begin{proposition}\label{prop101}
For a long virtual knot $K$, we have the following. 
\begin{enumerate}
\setlength{\itemsep}{1mm}
\item
$W(\widehat{K};t)=W_0(K;t)+W_1(K;t^{-1})$. 
\item
$I(\widehat{K};t)=F_{01}(K;t)+G_{00}(K;t)+G_{11}(K;t^{-1})+H_{01}(K;t^{-1})$. 
\item
$I\!I(\widehat{K};t)=F_{00}(K;t)+F_{11}(K;t)+
G_{01}(K;t)+G_{01}(K;t^{-1})$\smallskip\\
\hphantom{$I\!I_{\widehat{K}}(t)=$} 
$+G_{10}(K;t)+G_{10}(K;t^{-1})
+H_{00}(K;t)+H_{11}(K;t)$. 
\end{enumerate}
\end{proposition}

To prove Proposition~\ref{prop101}, 
we prepare the following lemma. 
Let $(\Sigma_{g},D)$ be a surface realization of $K$, 
and $c_1,\dots,c_n$ the real crossings of $D$. 
Recall that $I_a(D)$ is the set of indices of real crossings 
of type $a\in\{0,1\}$. 
The surface realization $(\Sigma_{g},D)$ can also be regarded as 
a surface realization $(\Sigma_{g},\widehat{D})$ 
of the closure $\widehat{K}$ 
by ignoring the basepoint of $D$. 
By abuse of notation, we also use $c_i$ 
to denote the corresponding crossing of $\widehat{D}$. 

\begin{lemma}\label{lem102}
For any integer $i\in\{1,\dots,n\}$, we have 
\[(\gamma_i,\overline{\gamma}_i)=
\begin{cases} 
(\alpha_i,\beta_i) & \mbox{for }i\in I_0(D), \\
(\beta_i,\alpha_i) & \mbox{for }i\in I_1(D). 
\end{cases}\]
\end{lemma}

\begin{proof}
See Figure~\ref{fig:cycles-knot-pf}, 
where ``$\bullet$'' on a dotted line denotes the basepoint of $D$. 
\end{proof}

\begin{figure}[htbp]
  \centering
    \begin{overpic}[]{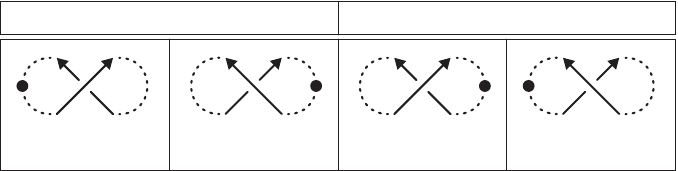} 
      \put(61,70){$i\in I_{0}(D)$}
      \put(224,70){$i\in I_{1}(D)$}
      \put(37,49){$c_{i}$}
      \put(19,39){$\overline{\gamma}_{i}$}
      \put(54,39){$\gamma_{i}$}
      \put(12,17){$\beta_{i}$}
      \put(61,18){$\alpha_{i}$}
      \put(29,5){$\e_{i}=1$}
      \put(118.5,49){$c_{i}$}
      \put(101,39){$\gamma_{i}$}
      \put(135.5,39){$\overline{\gamma}_{i}$}
      \put(94,18){$\alpha_{i}$}
      \put(142.5,17){$\beta_{i}$}
      \put(106,5){$\e_{i}=-1$}
      \put(200,49){$c_{i}$}
      \put(182,39){$\overline{\gamma}_{i}$}
      \put(217,39){$\gamma_{i}$}
      \put(175,18){$\alpha_{i}$}
      \put(224,17){$\beta_{i}$}
      \put(191,5){$\e_{i}=1$}
      \put(281,49){$c_{i}$}
      \put(263.5,39){$\gamma_{i}$}
      \put(298,39){$\overline{\gamma}_{i}$}
      \put(256.5,17){$\beta_{i}$}
      \put(305,18){$\alpha_{i}$}
      \put(269,5){$\e_{i}=-1$}
    \end{overpic}
  \caption{Proof of Lemma~\ref{lem102}}
  \label{fig:cycles-knot-pf}
\end{figure}

\begin{proof}[Proof of {\rm Proposition~\ref{prop101}}] 
We may assume that $D$ is untwisted. 

(i) By Lemma~\ref{lem102}, we have 
\begin{align*}
W(\widehat{K};t)
&=
\sum_{i=1}^n \varepsilon_i
(t^{\gamma_i\cdot\overline{\gamma}_i}-1)=
\sum_{i\in I_0}\varepsilon_i
(t^{\alpha_i\cdot\beta_i}-1)
+
\sum_{i\in I_1}\varepsilon_i
(t^{\beta_i\cdot\alpha_i}-1)\\
&=
W_0(K;t)+W_1(K;t^{-1}). 
\end{align*}

(ii) Since we have $\omega(\widehat{D})=0$, 
it follows from Lemma~\ref{lem102} that 
\begin{align*}
I(\widehat{K};t)
&=
\sum_{i,j=1}^n
\varepsilon_i\varepsilon_j(t^{\gamma_i\cdot\overline{\gamma}_j}-1)\\
&=\sum_{i,j\in I_0}\varepsilon_i\varepsilon_j(t^{\alpha_i\cdot\beta_j}-1)
+
\sum_{i\in I_0,\, j\in I_1}
\varepsilon_i\varepsilon_j(t^{\alpha_i\cdot\alpha_j}-1)\\
&\quad+
\sum_{i\in I_1,\, j\in I_0}
\varepsilon_i\varepsilon_j(t^{\beta_i\cdot\beta_j}-1)
+\sum_{i,j\in I_1}
\varepsilon_i\varepsilon_j(t^{\beta_i\cdot\alpha_j}-1)
\\
&=
g_{00}(D;t)+f_{01}(D;t)+h_{10}(D;t)+g_{11}(D;t^{-1})\\
&=
G_{00}(K;t)+F_{01}(K;t)+H_{10}(K;t)+G_{11}(K;t^{-1}).
\end{align*}
Since $H_{10}(K;t)=H_{01}(K;t^{-1})$ holds by Lemma~\ref{lem23}(i), 
we have the conclusion. 

(iii) Similarly to the proof of (ii), 
we have 
\begin{align*}
I\!I(\widehat{K};t)
&=
\sum_{i,j=1}^n
\varepsilon_i\varepsilon_j(t^{\gamma_i\cdot\gamma_j}-1)
+\sum_{i,j=1}^n
\varepsilon_i\varepsilon_j(t^{\overline{\gamma}_i\cdot\overline{\gamma}_j}-1)\\
&=\sum_{i,j\in I_0}\varepsilon_i\varepsilon_j(t^{\alpha_i\cdot\alpha_j}-1)
+
\sum_{i\in I_0,\, j\in I_1}
\varepsilon_i\varepsilon_j(t^{\alpha_i\cdot\beta_j}-1)\\
&\quad+
\sum_{i\in I_1,\, j\in I_0}
\varepsilon_i\varepsilon_j(t^{\beta_i\cdot\alpha_j}-1)
+\sum_{i,j\in I_1}
\varepsilon_i\varepsilon_j(t^{\beta_i\cdot\beta_j}-1)\\
&\quad+
\sum_{i,j\in I_0}\varepsilon_i\varepsilon_j(t^{\beta_i\cdot\beta_j}-1)
+
\sum_{i\in I_0,\, j\in I_1}
\varepsilon_i\varepsilon_j(t^{\beta_i\cdot\alpha_j}-1)\\
&\quad+
\sum_{i\in I_1,\, j\in I_0}
\varepsilon_i\varepsilon_j(t^{\alpha_i\cdot\beta_j}-1)
+\sum_{i,j\in I_1}
\varepsilon_i\varepsilon_j(t^{\alpha_i\cdot\alpha_j}-1)\\
&=
f_{00}(D;t)+g_{01}(D;t)+g_{01}(D;t^{-1})+h_{11}(D;t)\\
&\quad
+h_{00}(D;t)+g_{10}(D;t^{-1})+g_{10}(D;t)+f_{11}(D;t)\\
&=
F_{00}(K;t)+G_{01}(K;t)+G_{01}(K;t^{-1})+H_{11}(K;t)\\
&\quad
+H_{00}(K;t)+G_{10}(K;t^{-1})+G_{10}(K;t)+F_{11}(K;t).
\end{align*}
\end{proof}

From some known properties of the writhe and intersection polynomials for virtual knots, 
we obtain the following relations 
among the values at $t=1$ of the first and second derivatives of 
the polynomials for a long virtual knot. 

\begin{theorem}\label{thm103}
For a long virtual knot $K$, we have the following. 
\begin{enumerate}
\setlength{\itemsep}{1mm}
\item
$W_0'(K;1)=W_1'(K;1)$. 
\item
$G'_{00}(K;1)=G'_{11}(K;1)=0$. 
\item 
$G'_{01}(K;1)+G'_{10}(K;1)=0$. 
\item
$F_{01}'(K;1)=H_{01}'(K;1)$ and 
$F_{10}'(K;1)=H_{10}'(K;1)$. 
\item 
$F''_{00}(K;1)+F''_{11}(K;1)+H''_{00}(K;1)+H''_{11}(K;1)\equiv0\pmod{4}$. 
\end{enumerate}
\end{theorem}

\begin{proof}
We may assume that $D$ is untwisted; that is, 
$\sum_{i\in I_0}\e_i=\sum_{j\in I_1}\e_j=0$. 

(i) It was proved in~\cite[Proposition~4.2]{ST} that 
$W'(\kappa;1)=0$ holds for any virtual knot $\kappa$. 
Therefore, the equation follows from Proposition~\ref{prop101}(i).

(ii) It holds that 
\begin{align*}
G'_{00}(K;1)&=g'_{00}(D;1)
=\sum_{i,j\in I_0}\e_i\e_j(\alpha_i\cdot\beta_j)\\
&=\Biggl(\sum_{i\in I_0}\e_i\alpha_i\Biggr)
\cdot\Biggl(\sum_{j\in I_0}\e_j\beta_j\Biggr)
=\Biggl(\sum_{i\in I_0}\e_i\alpha_i\Biggr)
\cdot\Biggl(\sum_{j\in I_0}\e_j(\gamma_D-\alpha_j)\Biggr)\\
&=\left(\sum_{i\in I_0}\e_i\alpha_i\right)
\cdot\Biggl(\biggl(\sum_{j\in I_0}\e_j\biggr)\gamma_D-\sum_{j\in I_0}\e_j\alpha_j\Biggr)\\
&=-\left(\sum_{i\in I_0}\e_i\alpha_i\right)
\cdot\Biggl(\sum_{j\in I_0}\e_j\alpha_j\Biggr)=0.
\end{align*}
By Theorem~\ref{thm52}(i), we have $G_{11}(K;t)=G_{00}(K^\#;t)$, 
and hence $G'_{11}(K;1)=0$. 

(iii) 
Since it holds that 
\[
\alpha_{i}\cdot\beta_{j}+\alpha_{j}\cdot\beta_{i}
=\alpha_{i}\cdot(\gamma_D-\alpha_{j})+\alpha_{j}\cdot(\gamma_{D}-\alpha_{i})
=\alpha_{i}\cdot\gamma_{D}+\alpha_{j}\cdot\gamma_{D}
\]
for any $i,j$, 
we have 
\begin{align*}
&G'_{01}(K;1)+G'_{10}(K;1)
=g'_{01}(D;1)+g'_{10}(D;1) \\
&=\sum_{i\in I_{0},\,j\in I_{1}}\e_{i}\e_{j}(\alpha_{i}\cdot\beta_{j})
+\sum_{i\in I_{1},\,j\in I_{0}}\e_{i}\e_{j}(\alpha_{i}\cdot\beta_{j})
\\
&=\sum_{i\in I_{0},\,j\in I_{1}}\e_{i}\e_{j}(\alpha_{i}\cdot\beta_{j}+\alpha_{j}\cdot\beta_{i})
 \\
&=\sum_{i\in I_{0},\,j\in I_{1}}\e_{i}\e_{j}(\alpha_{i}\cdot\gamma_{D}
+\alpha_{j}\cdot\gamma_{D})
 \\%
&=\Biggl(\sum_{j\in I_{1}}\e_{j}\Biggr)
\Biggl(\sum_{i\in I_{0}}\e_{i}(\alpha_{i}\cdot\gamma_{D})\Biggr)
+\Biggl(\sum_{i\in I_{0}}\e_{i}\Biggr)
\Biggl(\sum_{j\in I_{1}}\e_{j}(\alpha_{j}\cdot\gamma_{D})\Biggr)=0.
\end{align*}

(iv) 
It was proved in~\cite[Theorem~3.2]{HNNS3} that $I'(\kappa;1)=0$ holds 
for any virtual knot~$\kappa$. 
By Proposition~\ref{prop101}(ii), 
we have 
\[F_{01}'(K;1)+G_{00}'(K;1)-G_{11}'(K;1)-H_{01}'(K;1)=0.\]
Therefore, it follows from (ii) that $F_{01}'(K;1)=H_{01}'(K;1)$. 
Furthermore, 
we have $F_{10}'(K;1)=H_{10}'(K;1)$ by Theorem~\ref{thm52}(i). 

(v)
We put 
\[
P(K;t)=G_{01}(K;t)+G_{10}(K;t) \text{ and } Q(K;t)=P(K;t)+P(K;t^{-1}). 
\]
Then we have $P(K;1)=0$ by definition 
and $P'(K;1)=0$ by (iii). 
Therefore, it follows from \cite[Lemma~3.5]{HNNS3} that 
$Q''(K;1)\equiv0\pmod{4}$. 
Since $I\!I''(\kappa;1)\equiv0\pmod{4}$ holds 
for any virtual knot~$\kappa$~\cite[Theorem~3.3]{HNNS3}, 
we have 
\begin{align*}
&F''_{00}(K;1)+F''_{11}(K;1)+H''_{00}(K;1)+H''_{11}(K;1)+Q''(K;1)\\
&\equiv F''_{00}(K;1)+F''_{11}(K;1)+H''_{00}(K;1)+H''_{11}(K;1)
\equiv0\pmod{4} 
\end{align*}
by Proposition~\ref{prop101}(iii). 
\end{proof}



\end{document}